\documentclass[12pt]{article}
\usepackage[T1]{fontenc}
\usepackage[dvips]{graphicx}
\graphicspath{{images/}}
\usepackage{verbatim}
\usepackage{amsmath,amsthm}
\usepackage{xspace}
\usepackage{pifont}
\usepackage{graphicx}
\usepackage{amssymb}
\usepackage{epic, eepic}
\usepackage{dsfont}
\usepackage{amssymb}
\usepackage{makeidx}
\usepackage{mathrsfs}
\usepackage{exscale}
\usepackage{color} 
\usepackage{overpic} 
\usepackage{bm}
\usepackage{bbm}
\usepackage{booktabs} 
\usepackage{color, colortbl}
\usepackage{subcaption}
\usepackage[numbers]{natbib}

\RequirePackage[colorlinks,citecolor=blue,urlcolor=blue]{hyperref}

\definecolor{Gray}{gray}{0.9}

\usepackage{amsmath,afterpage}
\usepackage{epsf}
\usepackage{graphics,color}
\RequirePackage[colorlinks,citecolor=blue,urlcolor=blue]{hyperref}
\usepackage[export]{adjustbox}

\usepackage{tcolorbox}

\def\0{\mathbf{0}}

\def\eps{\varepsilon}

\def\lam{\lambda}
\def\rr{\rightarrow}

\def \< {\langle}
\def \> {\rangle}

\def\ol{\overline}

\def\beqa{\begin{eqnarray}}
\def\eeqa{\end{eqnarray}}
\def\beqas{\begin{eqnarray*}}
\def\eeqas{\end{eqnarray*}}

\newtheorem{theorem}{Theorem}[section]
\newtheorem{lemma}[theorem]{Lemma}

\newtheorem{proposition}[theorem]{Proposition}

\newtheorem{corollary}[theorem]{Corollary}

\newtheorem{remark}[theorem]{Remark}
\newtheorem{example}[theorem]{Example}

\newtheorem{definition}[theorem]{Definition}
\newtheorem{assumption}[theorem]{Assumption}
\numberwithin{equation}{section}
\newcommand{\hatd}[1]{{}}

\newcommand{\bd}{\begin{displaymath}}
\newcommand{\ed}{\end{displaymath}}
\newcommand{\be}{\begin{equation}}
\newcommand{\ee}{\end{equation}}
\newcommand{\bq}{\begin{eqnarray}}
\newcommand{\eq}{\end{eqnarray}}
\newcommand{\bn}{\begin{eqnarray*}}
\newcommand{\en}{\end{eqnarray*}}

\newcommand{\re}{\mathds{R}}

\def\wt{\widetilde}

\def\P{\mathbb{P}}
\def\E{{\mathbb{E}}}
\newcommand{\R}{{\mathbb R}}
\newcommand{\id}{{\rm id}}

\usepackage{todonotes}

\usepackage[normalem]{ulem}

\usepackage{mathtools}
\mathtoolsset{showonlyrefs}

\usepackage{authblk}

\title{Equilibrium in Functional Stochastic Games with Mean-Field Interaction}
\author[1]{Eduardo Abi Jaber\thanks{The first author is grateful for the financial support from the Chaires FiME-FDD, Financial Risks, Deep Finance \& Statistics and Machine Learning and systematic methods in finance at Ecole Polytechnique.}}
\author[2]{Eyal Neuman}
\author[3]{Moritz Vo{\ss}}
\affil[1]{Ecole Polytechnique, CMAP}
\affil[2]{Department of Mathematics, Imperial College London}
\affil[3]{Department of Mathematics, University of California Los Angeles}

\begin{document}

 \vspace{-0.5cm}
\maketitle

\begin{abstract}
We consider a general class of finite-player stochastic games with mean-field interaction, in which the linear-quadratic cost functional includes linear operators acting on controls in $L^2$. We propose a novel approach for deriving the Nash equilibrium of the game {semi-}explicitly in terms of operator resolvents, by reducing the associated first order conditions to a system of stochastic Fredholm equations of the second kind and deriving their solution {in semi-explicit form}. 
Furthermore, by proving stability results for the system of stochastic Fredholm equations, we derive the convergence of the equilibrium of the $N$-player game to the corresponding mean-field equilibrium. As a by-product, we also derive an $\eps$-Nash equilibrium for the mean-field game, which is valuable in this setting as we show that the conditions for existence of an equilibrium in the mean-field limit are less restrictive than  in the finite-player game. Finally, we apply our general framework to solve various examples, such as stochastic Volterra linear-quadratic games, models of systemic risk and advertising with delay, and optimal liquidation games with transient price impact. 
\end{abstract} 

\begin{description}
{\small \item[Mathematics Subject Classification (2010):] 49N80, 49N90, 93E20, 91G80 
\item[JEL Classification:] C02, C61, C73, G11, G32

\item[Keywords:] mean-field games, Nash equilibrium, Volterra stochastic control, optimal portfolio liquidation, systemic risk, price impact}
\end{description}

\bigskip

\newpage 

\section{Introduction} 

Large population stochastic games and their mean-field limits have attracted a considerable attention in the past decades since the pioneering work of \citet{LasryLions} and \citet*{huang2006large}, which was further developed by {\citet{Car-del-book,Car-del-book2}, \citet{Cardaliaguet-book,Fischer:2017wc}, among others.} The tractability and convergence properties of this versatile class of {models} has applications in various research areas such as mathematical finance, economics, population modeling, and marketing (see, e.g., Chapter 1 of \cite{Car-del-book}). 
One of the main challenges in the area of finite population stochastic games is to derive explicitly  the Nash equilibrium of the system. Motivating examples from mathematical finance include price impact games with competition for liquidity between agents~\citep*{voss.19, DrapeauLuoSchiedXiong:19,EvangelistaThamsten:20,neuman2021trading,MK-M-N-2022}, systemic risk games introduced by \citet{carmona2013mean,carmona2018systemic,fouque2018mean}, as well as optimal investment problems studied in~\citet{Lacker-Zariphopoulou, Lacker:2020aa}. In various extensions of the aforementioned models it turns out that the state variables of the players and their objective functionals naturally depend on the entire trajectory of the controls. These generalizations give rise to infinite-dimensional dynamic stochastic games. In such a setting, deriving a Nash equilibrium of the system is in general considered to be intractable, and solutions to such games only appear in very particular examples; see Section~\ref{sec:examples} for a survey of such problems.

In this work we develop a novel approach for solving a general class of infinite-dimensional stochastic games with mean-field interaction between the players. Specifically, we consider $N$-player stochastic games in which each agent $i$ has an objective functional of the form
\be\label{eq:J^iintro2}
\begin{aligned}
    J^{i}(u^i) := & \, \E\left[ -\langle \bar u,\boldsymbol A_1  \bar u \rangle_{L^2} - \langle u^i,\boldsymbol A_2 u^i \rangle_{L^2} - \langle u^i, (\boldsymbol{A}_3 + \boldsymbol{A}^*_4)  \bar u \rangle_{L^2} \right. \\ 
    & \hspace{14pt} \left. + \langle b^i, u^i \rangle_{L^2} +\langle b^0, \bar u\rangle_{L^2} + c^i \right],
\end{aligned}
\ee 
where $u^i$ represents the agent's control and $\bar u = N^{-1}\sum_{i=1}^N u^{i}$ captures a mean-field interaction between all agents. Here, the symbols $\boldsymbol A_i$, $i=1, 2, 3, 4$ denote non-anticipative linear operators on $L^2([0,T], \R)$; $(b^i_s)_{s\in [0,T]}, (b^0_s)_{s\in [0,T]}$ are progressively measurable stochastic processes {which encode individual and common signals}; and $c^i$ is a random variable. The inner product is defined in the usual sense as $\langle f,g\rangle_{L^2}:= \int_0^T f(s) g(s) ds$ for $f,g \in L^2([0,T], \R)$. 

The main goal of this work is to derive the Nash equilibrium {in semi-explicit form} of finite-player games with reward functionals~\eqref{eq:J^iintro2}. In order to achieve this, we {propose} a new technique to solve simultaneously for all agents $i \in \{1,\ldots,N\}$ their individual path-dependent stochastic control problem
\begin{equation} \label{intro:opt}
{\underset{u^i}{\text{maximize}} \; J^{i}(u^{i})}
\end{equation}
over progressively measurable and square-integrable $u^i$'s. Moreover, we also apply our method to solve the corresponding simpler {limit} mean-field game when the number of agents $N$ tends to infinity and provide a convergence result of the finite-player equilibrium strategies toward their mean-field limits. 

Our approach uses variational calculus in order to establish sufficient first order conditions for a Nash equilibrium which take the form of stochastic Fredholm equations with both forward and backward components. The derivation of the equilibrium then relies on solving this system of equations for which we develop a new method of solution. The approach is versatile enough to allow for non-Markovian and non-semimartingale settings with no additional effort, in the sense that the controls are allowed to depend on the entire trajectories of the inputs $(b_i)_{i=0}^N$ which are only assumed to be progressively measurable; that is, neither necessarily semimartingales nor independent. Our framework unifies and extends {several} examples of stochastic dynamical games that have appeared in the literature beyond the Markovian and semimartingale case and were often considered as intractable (see Section \ref{sec:examples}). In the following, we give a concrete summary of the contribution of the methods developed in this paper. 

\textbf{Our contributions.} In Theorems~\ref{thm:opt_ubar} and~\ref{thm:main-finite}, we show that in equilibrium the average control $\bar u$ is characterized via the stochastic Fredholm equation
\begin{align}\label{eq:introVolterra}
    \bar u_t = f_t - \int_0^t K(t,r) \bar u_r dr -\int_t^T L(r,t) \mathbb E_t \bar u_r dr, \quad t\in [0,T],
\end{align}  
where $f$ is progressively measurable and $K,L$ are deterministic kernels which are determined by the inputs of the model $(\boldsymbol A_1, \boldsymbol A_2, \boldsymbol A_3, \boldsymbol A_4, b^i,b^0,c^i)$.
One of our main {contributions} is that we derive a {semi-explicit} solution to \eqref{eq:introVolterra} using a {new} approach (see Proposition~\ref{L:FredholmConditional}). Then, we utilize the solution to \eqref{eq:introVolterra} to disentangle and solve the optimisation problem of each player and hence to derive the Nash equilibrium. In addition, we derive a stability result for \eqref{eq:introVolterra} (see Proposition~\ref{L:Fredholm_conv}). This stability result is the crucial ingredient for deriving  the convergence of the finite-player game equilibrium toward its mean-field game equilibrium limit in Theorem~\ref{thm:convergence}.

The framework developed in this paper is {distinct} in that the solvability, stability, and consistency of finite-player games with mean-field interaction boil down to the study of the stochastic Fredholm equation \eqref{eq:introVolterra}. To the best of our knowledge, our approach gives the first canonical method for deriving {semi-explicitly} the unique Nash equilibrium to this general class of stochastic games.  The closest result to Theorems \ref{thm:opt_ubar} and \ref{thm:main-finite} appeared in \citet{huang2015mean}, where the authors derived a first order condition for a special case of \eqref{eq:J^iintro2} in terms of a system of stochastic integral equations, but they did not obtain a solution to the system. In the context of linear–quadratic optimal control problems for stochastic Volterra equations, \citet{wang2018linear}  characterized and solved the problem in terms of stochastic Fredholm equations, and established the existence of solutions to the Fredholm equation via fixed-point methods (see  \citet[Theorem 4.3]{wang2018linear}). However, no explicit solutions were obtained in that work. More recently, similar equations appeared in \citet{hamaguchi2023maximum} and were solved for the specific case of a fractional kernel and a Gaussian signal. We also refer to \citet{bensoussan2017linear} where first order conditions for linear-quadratic stochastic games with delays in the state and the control were derived. In Section \ref{S:delaystate} we show that we can compute semi-explicit solutions to this class of games as a corollary of our main results.  

In Section~\ref{sec:examples}, we highlight the versatility of the objective functional in \eqref{eq:J^iintro2} by showing that it nests {various} examples of dynamical stochastic games beyond the Markovian and semimartingale case. In particular, we introduce in Section~\ref{S:Volterragame} a general class of Linear-Quadratic Stochastic Volterra games whose objective functional is shown to be equivalent to \eqref{eq:J^iintro2}. Then, we show how this framework accommodates, extends and {semi-explicitly} solves three {popular} examples that appeared in the literature: (i) inter-bank lending and borrowing models with delay in the control which were studied in \citet{carmona2018systemic,fouque2018mean} (see Section~\ref{S:exsystematic}); (ii) advertising models \`a la~\citet{nerlove1962optimal, gozzi245stochastic} with mean-field effect and delay in the state (Section~\ref{S:delaystate}); (iii) multi-player price impact games with a general propagator in the spirit of~\citet{abi2022optimal, neuman2021trading} (Section~\ref{S:marketimpact}).

Convergence results of the finite-player Nash equilibrium to the corresponding mean-field equilibrium in Markovian settings have attracted considerable attention recently (see \citet{Cardaliaguet-book} for the main reference on this topic). In particular \citet{Laur20} proved such convergence results for open-loop equilibria of games with idiosyncratic noise for each of the players.~\citet{neuman2021trading} studied the corresponding problem for portfolio liquidation games with common noise. \citet{lacker20} and~\citet{Djete21} proved the convergence of closed-loop solutions (under the a priori assumption that they exist) to the mean-field solution in the case where each player is influenced by idiosyncratic and common noise. In these papers (except for~\cite{neuman2021trading}) the presence of idiosyncratic noise is crucial to establish the convergence. In our specific setup, of semi-explicitly solvable linear-quadratic mean-field games, we provide in Theorems~\ref{thm:meanfieldgame} and~\ref{thm:meanfieldgame_inf} the equilibrium of the infinite-player mean-field game in our non-Markovian setting and show that the conditions to obtain the latter are less restrictive than in the corresponding finite $N$-player game in Theorems~\ref{thm:opt_ubar} and \ref{thm:main-finite}. Moreover, in Theorem~\ref{thm:convergence} we derive the aforementioned convergence which does not require the presence of idiosyncratic noise. Indeed, we show that convergence is obtained as a corollary of a stability result for the associated stochastic Fredholm equation (see Proposition~\ref{L:Fredholm_conv}). For the sake of completeness, we also provide in Theorem~\ref{thm-eps-nash} the associated $\eps$-Nash equilibrium result for our class of infinite-dimensional games. In other words, we show that the mean-field equilibrium strategy gives close to optimal rewards for each player in the $N$-player game when $N$ is sufficiently large. 

Finally, let us illustrate in the following toy example that our choice of objective functionals of type~\eqref{eq:J^iintro2} also includes games in which the cost functional explicitly contains state variables. We refer to Section~\ref{sec:examples} for the generalization of this equivalence to functional games as in \eqref{eq:J^iintro2}.

\textbf{Motivating example.} Let us consider a simple stochastic differential Markovian $N$-player game with controlled state variables given by
\begin{align} \label{toy-xi}
dX^i_t = \left( a X^i_t + u^i_t \right)dt + dW^i_t, \quad X^i_0 = 0, \quad i=1,\ldots, N,
\end{align} 
where $u^i$ is the control of the $i$-th player, $a \in \R$  and  $(W^1, \ldots, W^N)$ is an $N$-dimensional Brownian motion. Each player $i \in \{1,\ldots,N\}$ wishes to maximize following objective functional with mean-field interaction
\begin{equation} \label{toy-j} 
    J^i(u^i) = \mathbb E\left[ \int_0^T \left(- (u^i_t)^2 + {u^i_t} \left( \frac{1}{N} \sum_{j=1}^N X^j_t \right) \right) dt  + (X^i_T)^2 \right].
\end{equation}
We will show how this game can be transformed into a game of type~\eqref{eq:J^iintro2}. Indeed, note that we can rewrite $(X^i_T)^2$ by first using the variation of constants formula to solve for $X^i$ in~\eqref{toy-xi}, i.e., 
\begin{align*}
    X^i_t = \int_0^t e^{ a(t-s)} u_s^i ds + \int_0^t e^{a(t-s)}  dW_s^i,
\end{align*}
and then apply Fubini’s theorem to obtain
\begin{align}
    (X^i_T)^2 &= 2\int_0^T u^i_s \left(\int_0^s  e^{a(2T-s-r)}  u^i_r dr \right)
  ds  \\
  &\quad  \quad + 2 \int_0^T  \left( \int_0^T e^{a(2T-s-r)}  dW^i_r \right) u^i_s ds + \left( \int_0^T e^{a(T-s)} dW_s^i \right)^2.
\end{align} 
This shows that the objective functional \eqref{toy-j} can be written in the form
  \begin{align}\label{eq:Jintro}
      J^i(u^i) = \mathbb \E\left[  - \langle u^i,\boldsymbol A_2 u^i \rangle_{L^2} +{\langle u^i, \boldsymbol A_3 \bar u\rangle_{L^2}} + \langle b^i, u^i \rangle_{L^2}  + c^i \right],
  \end{align}
where
\begin{align*}
   (\boldsymbol A_2 u)(s) = & \, u_s - {2}\int_0^s e^{a(2T-s-r)} u_r dr, \quad  s\leq T, \\
   (\boldsymbol A_3 u)(s) = & \, \int_0^s e^{a(s-r)} u_r dr, \quad  s\leq T,
\end{align*}
and
\begin{align*}
    b^i_s = 2 \int_0^s e^{a(2T-s-r)} dW^i_r + \frac{1}{N} \sum_{j=1}^N \int_0^s e^{a(s-r)} dW^j_r, \quad c^i = \left( \int_0^T e^{a(T-s)} dW_s^i \right)^2,
\end{align*}
with $\bar u = \frac{1}{N} \sum_{j=1}^N u^j$. Hence, we observe that the objective functional in~\eqref{eq:Jintro} corresponds to a stochastic differential game without state variables. Moreover, this representation very naturally lends itself to directly employing first order variational calculus. This motivates us to study more general objective criteria in the spirit of~\eqref{eq:Jintro}. 

\textbf{Structure of the paper.} In Section~\ref{sec:FinitePlayer} we present the class of finite-player stochastic games and state our main result on the derivation of the Nash equilibrium. In Section~\ref{sec:examples} we describe multi-agent models that are included in our framework. Section~\ref{sec:MeanFieldGame} is dedicated to the derivation of the Nash equilibrium in the corresponding infinite-player mean-field game, as well as related $\eps$-Nash equilibrium and convergence results. In Section \ref{sec-fredholm} we derive some essential results on the solution and stability of stochastic Fredholm equations. Finally, Sections \ref{sec-proof-finite}--\ref{sec-proof-infinite} are dedicated to the proofs of our main results.

\section{The Finite-Player Game} \label{sec:FinitePlayer}

In this section we derive the Nash equilibrium of \eqref{eq:J^iintro2} {in semi-explicit form}. Before stating this result, we introduce some essential definitions of function spaces and integral operators.   

\subsection{Function spaces, integral operators} 
We denote by $\langle \cdot, \cdot \rangle_{L^2}$ the inner product on $L^2([0,T], \R)$, that is 
\be \label{in-prod} 
\langle f, g\rangle_{L^2} = \int_0^T f(s) g(s) ds, \quad f,g\in L^2\left([0,T],\mathbb R\right),
\ee
where $\| \cdot\|_{L^2}$ is the induced norm. 
We define $L^2\left([0,T]^2,\mathbb R^{ }\right)$ to be the space of measurable kernels $G:[0,T]^2 \to \R$ such that 
\begin{align*} 
\int_0^T \int_0^T |G(t,s)|^2 dt ds < \infty.
\end{align*}
For any  kernel $G \in L^2\left([0,T]^2,\mathbb R^{}\right)$, we denote by {$\boldsymbol G$} the integral operator   induced by the kernel $G$, that is, 
\begin{align}\label{boldG-def} 
({\boldsymbol G} f)(s)=\int_0^T G(s,u) f(u)du,\quad f \in L^2\left([0,T],\mathbb R\right).
\end{align}
$\boldsymbol G$ is a linear bounded operator from  $L^2\left([0,T],\mathbb R \right)$ into itself. 
We denote by $G^*$ the adjoint kernel of $G$ for $\langle \cdot, \cdot \rangle_{L^2}$, that is 
\begin{align*} 
G^*(s,u) &= \; G(u,s), \quad  (s,u) \in [0,T]^2,
\end{align*}
and by $\boldsymbol{G}^*$ the corresponding adjoint integral operator.  

\subsection{Definition of the $N$-player game} \label{subsec-model-def} 

We present the class of functional stochastic games which are studied in this paper. Let $T>0$ denote a finite deterministic time horizon and let $N\geq 2$ be an integer. We fix a filtered probability space $(\Omega, \mathcal F, \mathbb F:=(\mathcal F_t)_{0 \leq t \leq T}, \mathbb P)$ satisfying the usual conditions of right continuity and completeness and use the notation $\mathbb E_t = \mathbb E [ \cdot | \mathcal F_t]$ to represent the conditional expectation with respect to $\mathcal{F}_t$. 

For $i=0,\ldots,N$, let $c^i$ be a random variable and $b^i= (b^i_t)_{0 \leq t\leq T}$ be progressively measurable processes satisfying
\be \label{ass:P} 
  \int_0^T \mathbb E\left[(b_s^i)^{2} \right] ds < \infty, \quad \mathbb E[c^i] < \infty, \quad i=0,\ldots,N.
\ee

We say that a measurable Volterra kernel $G:[0,T]^2 \to \R$, i.e., such that $G(t,s)=0$ whenever $s \geq t$, is nonnegative definite if for every $f\in L^2\left([0,T],\mathbb R\right)$ we have  
\be \label{pos-def}
\int_{0}^T\int_{0}^T\big(G(t,s)+ G(s,t) \big)f(s)f(t)dsdt \geq 0. 
\ee

Next we define the class of Volterra kernels which will be used in our setting.  
\begin{definition}[Class of admissible kernels $\mathcal G$]  \label{def-ker-admis} 
We say that a nonnegative definite Volterra kernel $G:[0,T]^2 \mapsto \R_{+}$ is in the class of kernels $\mathcal G$ if it satisfies the following conditions:  
\begin{align}\label{eq:assumtionG}
 	\sup_{t\leq T} \int_0^T |G(t,s)|^2 ds  + \sup_{s\leq T} \int_0^T |G(t,s)|^2 dt< \infty. 
 \end{align}  
\end{definition} 

\begin{definition}[Admissible Volterra operator] \label{def-admis-op}
We say that an integral operator {$\boldsymbol G$} is an admissible Volterra operator if {the kernel $G$ by which it is induced is in $\mathcal G$}. 
\end{definition} 
\begin{remark} \label{rem:opNonNeg}
Observe that if $G \in \mathcal{G}$, property~\eqref{pos-def} implies for any $f\in L^2\left([0,T],\mathbb R\right)$, 
\begin{equation}
     \langle f, \boldsymbol G f\rangle_{L^2} = \frac{1}{2}\langle f, (\boldsymbol G + \boldsymbol G^*) f \rangle_{L^2}  \geq 0.
\end{equation}
\end{remark}

We consider $N$ agents $i \in \{1, \ldots, N\}$ who select their controls $u^{i,N}$ from the admissible set  
\be \label{def:admissset} 
\mathcal U :=\left\{ u \, : \, \mathbb{F}\textrm{-progressively measurable s.t. } \int_0^T \mathbb E\left[u_s^2 \right] ds  <\infty \right\}
\ee
and let
\begin{equation} \label{def:uNotation}
\begin{aligned}
 u^{-i,N} := & \, (u^{1,N},\ldots,u^{i-1,N}, u^{i+1,N}, \ldots, u^{N,N}), \\
 \bar u^N := & \, \frac{1}{N} \sum_{j=1}^N u^{j,N}, \qquad \bar u^{-i,N} := \frac{1}{N} \sum_{\substack{j=1 \\ j\neq i}}^N u^{j,N}.
 \end{aligned}
\end{equation}

In the following, we consider operators $(\boldsymbol{{A}}_1, \boldsymbol{{A}}_2, \boldsymbol{{A}}_3, \boldsymbol{{A}}_4)$ which satisfy the following assumptions. 

\begin{assumption} \label{assum-op} 
We assume that $\boldsymbol{A}_1$, $\boldsymbol{A}_3$, and $\boldsymbol{{A}}_4$ are admissible Volterra operators, and $\boldsymbol{A}_2$ is given by 
\be \label{b-bar} 
\boldsymbol{A}_2 := \lambda \id + \boldsymbol{\hat{A}}_2 
\ee
where $\boldsymbol{ \hat{A}}_2$ is an admissible Volterra operator, $\lambda>0$ is a constant and $\id$ is the identity operator, i.e.~
$$
(\id f)(t)=f(t), \quad  \textrm{for all } t \in [0,T], \, f\in L^2\left([0,T],\mathbb R\right).$$
\end{assumption}

Each agent $i \in \{1, \ldots, N\}$ has the following individual performance functional 
\be\label{eq:J^i}
\begin{aligned}
    J^{i,N}(u^{i,N}; u^{-i,N }) := & \, \E\left[ -\langle \bar u^N,\boldsymbol A_1  \bar u^N \rangle_{L^2} - \langle u^{i,N},\boldsymbol A_2 u^{i,N} \rangle_{L^2} - \langle u^{i,N},(\boldsymbol{A}_3 + \boldsymbol{A}_4^*) \bar u^N \rangle_{L^2} \right. \\ 
    & \hspace{14pt} \left. + \langle b^i, u^{i,N} \rangle_{L^2} +\langle b^0, \bar u^N\rangle_{L^2} + c^i \right], 
\end{aligned}
\ee
where $\bar u^N \in \mathcal{U}$ describes the mean-field interaction between the agents.

\begin{remark} \label{rem:bi}
Note that in the objective functional in~\eqref{eq:J^i} the processes $(b^i_t)_{0\leq t \leq T}$, $i=0,\ldots,N$, are fairly general. They can be thought of as incorporating simultaneously different sources of noise: i.e., a common noise affecting all $N$ players as well as player $i$'s independent individual source of noise; and also other random or deterministic factors idiosyncratic to player~$i$. We refer to Section~\ref{sec:examples} below for specific examples.
\end{remark}

Our main goal in this section is to solve simultaneously for each agent $i \in \{1,\ldots,N\}$ their individual optimal stochastic control problem
\begin{equation} \label{def:FPGoptimization}
{\underset{u^{i,N} \in \mathcal{U}}{\text{maximize}} \; J^{i,N}(u^{i,N};u^{-i,N}). }
\end{equation}
This solution will establish a Nash equilibrium in the following usual sense.
\begin{definition} \label{def:Nash}
A set of strategies $\hat{u}^N = (\hat{u}^{1,N},\ldots,\hat{u}^{N,N}) \in \mathcal{U}^N$ is called {a  Nash equilibrium} if for all $i \in \{1,\ldots, N\}$ and for all admissible strategies $v \in \mathcal{U}$ it holds that
\begin{equation*}
    J^{i,N}(\hat{u}^{i,N};\hat{u}^{-i,N}) \geq J^{i,N}(v;\hat{u}^{-i,N}). 
\end{equation*}
\end{definition}

\subsection{Main results for the $N$-player game} 
For convenience, we introduce the following operators, 
\be \label{def:operator}
\begin{aligned} 
\boldsymbol{G} & :=\frac{\boldsymbol{A}_1 }{N^2} +\frac{\boldsymbol{A}_3 +  \boldsymbol{A}_4}{N} + \boldsymbol{\hat{A}}_2  , \\
\boldsymbol{H}& := \frac{\boldsymbol{A}_1 + \boldsymbol{A}^*_1}{N}  + \boldsymbol{A}_3 + \boldsymbol{A}^*_4. 
\end{aligned} 
\ee

Now we are ready to state our main results for the finite-player game. The first result derives {a semi-explicit} solution to the average control $\bar u^N$ in the Nash equilibrium of the game.  To formulate the result, we define for any $G \in \mathcal G$ the kernel
$$
G_t(s,r)=G(s,r) 1_{\{s,r\geq t\}}
$$
where $1_{\{\cdot\}}$ is the indicator set function. We denote by $ {\boldsymbol{G}}_t$ the associated Volterra operator. 

We also set
\begin{equation} \label{b-bar-def} 
    \bar b_t := \frac{1}{N} \sum_{i=1}^N b^i_t, \qquad 0 \leq t \leq T.
\end{equation}

\begin{theorem} \label{thm:opt_ubar} 
Under Assumption \ref{assum-op} and \eqref{ass:P} any Nash equilibrium $\hat{u}^N \in \mathcal{U}^N$ to the game \eqref{eq:J^i} is such that $\bar {\hat u}^N := \frac 1N \sum_{j=1}^N \hat{u}^{j,N}$ is given by
\be \label{eq:opt_ubar} 
\begin{aligned}
\bar {\hat u}^N_{t}  = \left((\id -  \ol {\boldsymbol B})^{-1} \ol {a} \right) (t), \qquad 0\leq t \leq T,
\end{aligned}
\ee
where $\ol {\boldsymbol B}$ is an integral operator of the form \eqref{boldG-def}  with kernel $\ol B$ given by
\begin{equation} \label{eq:FHubarCoeff}
\begin{aligned} 
\ol B(t,s) & :=  1_{\{s<t\}}\frac 1{2\lambda} \bigg( \left\langle   1_{\{t\leq \cdot\}} \ol L(\cdot,t),\ol{\boldsymbol D}_t^{-1}   1_{\{t\leq \cdot\}}  \ol K(\cdot,s)    \right  \rangle_{L^2}   -  \ol K(t,s)    \bigg), \\
\ol a_t & := \frac 1{2\lambda} \left(\bar b_{t} +\frac{1}{N} b^0_t  -  \left\langle  1_{\{t\leq \cdot\}} \ol L(\cdot,t), \ol{\boldsymbol D}_t^{-1} 1_{\{t\leq \cdot\}} \E_t \left[\bar{b}_{\cdot} + \frac{1}{N} b^0_{\cdot}\right] \right\rangle_{L^2} \right), \\
\ol K(t,s) & :=  \frac{N-1}{N} \left(\frac{A_1(t,s)}{N}+A_3(t,s)\right) + G(t,s) , \\ 
\ol L(t,s) & := \frac{N-1}{N} \left( \frac{A_1(t,s)}{N} + A_4(t,s) \right) + G(t,s), 
\end{aligned}
\end{equation}
and 
\bd\label{eq:barDt}
 \ol{\boldsymbol{D}}_t := 2 \lambda \id + \frac{N-1}{N} \boldsymbol{H}_t +  \boldsymbol{G}_t+ \boldsymbol{G}^*_t. 
\ed
for all $t \in [0,T]$.
\end{theorem}  

Using the result of Theorem \ref{thm:opt_ubar} we derive each player's Nash equilibrium strategy~$\hat{u}^{i,N}$.     
\begin{theorem} \label{thm:main-finite} 
Assume that \eqref{ass:P} and Assumption \ref{assum-op} are satisfied, and let $\bar {\hat u}^N$ be as in~\eqref{eq:opt_ubar}. The unique Nash equilibrium of the game \eqref{eq:J^i} is given by
\begin{equation} \label{eq:opt_ui}
\begin{aligned}
\hspace{-1pt} \hat{u}^{i,N}_{t}=\left((\id -  \boldsymbol B)^{-1} a^i \right) (t), \qquad 0\leq t \leq T,  
\end{aligned}
\end{equation}
where 
\begin{equation} \label{eq:FHuiCoeff} 
\begin{aligned}
a^i_t & := \frac 1{2\lambda} \bigg( b^i_{t} +\frac{1}{N} b^0_t - \E_t \left[(\boldsymbol{H} \bar{\hat u}^N)(t)\right]  \\
& \qquad  \qquad -  \left\langle{   1_{\{t\leq \cdot\}} \hat L(\cdot,t)},  \wt{\boldsymbol D}_t^{-1} 1_{\{t\leq \cdot\}} \E_t \left[b^i_{\cdot} +\frac{1}{N} b^0_{\cdot} - (\boldsymbol{H} \bar{\hat u}^N)(\cdot) \right] \right\rangle_{L^2} \bigg), \\
B(t,s) & :=  1_{\{s<t\}}\frac 1{2\lambda} \bigg(  \left\langle 1_{\{t\leq \cdot\}} \hat L(\cdot,t),  \wt {\boldsymbol D}_t^{-1}   1_{\{t\leq \cdot\}}     \hat K(\cdot, s) \right  \rangle_{L^2} -\hat K (t, s)    \bigg), \\
\hat K(t,s) & := G(t,s) - \frac{1}{N} \left(\frac{A_1(t,s)}{N}+A_3(t,s)\right), \\ 
\hat L(t,s) & := G(t,s)-  \frac{1}{N} \left( \frac{A_1(t,s)}{N} + A_4(t,s) \right),
\end{aligned}
\end{equation}
and
$$
 \wt{\boldsymbol{D}}_t :=  2\lambda \id - \frac{1}{N} \boldsymbol{H}_t + \boldsymbol{G}_t+ \boldsymbol{G}^*_t,
$$
for all $t \in [0,T]$.
\end{theorem}  
The proofs of Theorems~\ref{thm:opt_ubar} and~\ref{thm:main-finite} are given in Section \ref{sec-proof-finite}. 

\begin{remark}
Our assumptions in Theorems \ref{thm:opt_ubar} and \ref{thm:main-finite} ensure that the operators $\{\ol{\boldsymbol{D}}_t\}_{t\in [0,T]}$, $\{\wt{\boldsymbol{D}}_t\}_{t\in [0,T]}$, $(\id - \ol{ \boldsymbol B})$ and $(\id -   {\boldsymbol B})$ are invertible so that $(\bar{\hat{u}}^N, \hat{u}^{i,N})$ in  \eqref{eq:opt_ubar} and \eqref{eq:opt_ui} are well defined.  
\end{remark}

\begin{remark} 
The results of Theorems~\ref{thm:opt_ubar} and~\ref{thm:main-finite} improve various results for linear-quadratic (LQ) stochastic mean-field games that have appeared recently in the literature.
In \citep{hamaguchi2022linear, huang2015mean, wang2022linear, wang2015linear} a class of LQ stochastic control problems and games with state variables that satisfy Volterra integral equations were studied. The solutions were  described in terms of a system of stochastic Volterra or Fredholm equations, but no explicit solutions were obtained. In \cite{hamaguchi2023maximum} they were solved for the specific case of a fractional kernel and a Gaussian signal. Our results in this section provide {semi-explicit} solutions to this class of problems for general kernels and signals; see Section \ref{S:Volterragame} for additional details. In~\cite{gozzi2015stochastic} a class of stochastic control problems with delay in the controls was studied. The solution to the control problem was characterised in terms of a system of integral equations. Our framework  provides a solution for a single agent and multi-player games with delay in the controls, in the states and in the mean-field interaction. Our framework also allows common and non-semimartingale noise which are not straightforward extensions in these models. See Sections \ref{S:exsystematic} and \ref{S:delaystate} for an elaborate discussion on our contribution to stochastic games with delays in the controls and the states. 
\end{remark}

\begin{remark} 
In Section \ref{sec:examples}  we apply our general framework in order to solve various {popular} models, such as stochastic Volterra linear-quadratic games, models of systemic risk, advertising with delay, and optimal liquidation games with transient price impact. As described in Section \ref{sec:examples}, many of the models in these classes do not have a closed form solution, and our framework establishes a new approach to deriving {semi-explicit solutions}. See Remarks \ref{rem-imp3}, \ref{rem-imp1} and \ref{rem-imp2} for additional details. 
\end{remark}

\begin{remark}\label{R:num}
 The  expressions \eqref{eq:opt_ui} and \eqref{eq:opt_ubar} that determine the optimal strategy $\hat{u}^{i}$ lend themselves
naturally to numerical discretization schemes using the so-called Nyström method where the time interval $[0,T]$ is discretized and the operators/kernels are approximated by corresponding matrices; see for instance~\cite[Section 5]{abi2022optimal} and the Jupyter notebook cited therein for an implementation in a single player case. 
\end{remark}

\begin{remark}
While the model presented in Section \ref{subsec-model-def} is rather general with respect to the operators and noise processes that appear in the cost functionals, the proofs of Theorems~\ref{thm:opt_ubar} and~\ref{thm:main-finite} rely heavily on tools from the theory of Volterra equations on $\mathbb{R}$. Therefore, the extension of our methods to controls taking values in discrete state spaces requires a new set of tools and can be of interest for future research. 
\end{remark}

\begin{remark}
Similar to~\eqref{b-bar}, one can also consider $\boldsymbol{A}_i$ in~\eqref{eq:J^i} for $i\in \{1, 3, 4\}$ to be of the form
\begin{equation*}
    \boldsymbol{A}_i := \lambda_i \id + \boldsymbol{\hat{A}}_i     
\end{equation*}
with constants $\lambda_i \geq 0$ and admissible Volterra operators $\boldsymbol{\hat{A}}_i$. The approach developed in this paper remains applicable; the formulas only become slightly more cumbersome. Therefore, we refrain from including this variant in our analysis. The crucial ingredient is to require $\boldsymbol{A}_2$ to be of the form~\eqref{b-bar} with $\lambda > 0$.
\end{remark}

\section{{Illustrative Examples}} \label{sec:examples} 

In this section, we showcase the applicability of our results from Theorems \ref{thm:opt_ubar} and \ref{thm:main-finite} by illustrating how they can be used to address various instances of stochastic dynamical games which have been proposed in the literature.

We first introduce in Section~\ref{S:Volterragame} a generic framework of a Linear-Quadratic Stochastic Volterra game whose objective functional is shown to be of the form~\eqref{eq:J^i}. Then, we show how such framework accommodates and extends three models which have been studied in the literature: (i) optimal liquidation games with general transient price impact kernels (also known as propagators) and signals (Section~\ref{S:marketimpact}); (ii) an inter-bank lending and borrowing model with delay in the control (Section~\ref{S:exsystematic}); (iii) advertising models with mean-field effect and delay in the state (Section~\ref{S:delaystate}). For the sake of readability, we omit the superscript $N$ from $u^{i,N}, \bar u^N$ and $u^{-i,N}$ throughout this section.

\subsection{Stochastic Volterra Linear-Quadratic games}\label{S:Volterragame}
We show that generic stochastic Volterra linear-quadratic games are included in the general framework which was developed in Section \ref{sec:FinitePlayer}. 

 Define the following controlled Volterra state variables
\begin{align}
    X^i_t &= P^i_t + \int_0^t G_2 (t,s)u^i_s ds + \int_0^t G_3(t,s) \bar u_s ds, \quad i=1,\ldots, N, \label{eq:volterra1} \\
        Y^i_t &= R^i_t + \int_0^t G_1(t,s) \bar u_s ds, \quad i=1,\ldots, N,\label{eq:volterra2}
\end{align}
for some {real-valued} progressively measurable processes $P^i,R^i$  and Volterra kernels $G_1,G_2, G_3$. 
{The $\R^2$-valued process $Z^i= (X^i, Y^i)^\top$  is given by} 
\begin{align}\label{eq:volterraZi}
    Z^i_t = d^i_t + \int_0^t D(t,s) \begin{pmatrix}
         u^i_s \\
         \bar u_s
    \end{pmatrix}ds,
\end{align}
with 
\begin{align*}
d^i_t = \begin{pmatrix}
    P^i_t \\
    R^i_t 
\end{pmatrix} \quad \mbox{and} \quad
D(t,s) =\begin{pmatrix}
 {G_2}(t,s) & G_3(t,s) \\
0 & G_1(t,s) 
\end{pmatrix}.
\end{align*}
Consider the following objective functional to be maximized by the $i$-th player:  
\begin{align}\label{eq:volterraJ}
    \mathcal J^i_{Vol}(u^i, u^{-i}) =\mathbb E\left[ \int_0^T f^{Vol}_{i}(Z^i_t , u^i_t) dt + g^{Vol}_i(Z^i_T) \right],
\end{align}
where the running and terminal criterion have a linear-quadratic dependence in $(Z^i,u^i)$ of the form
\begin{align}
f^{Vol}_{i}( z, u) &= - p u^2 - z^\top Q z + u z^\top q,  \label{eq:volterraf}  \\
g^{Vol}_{i}( z ) &= - z^\top S z + z^\top s^i,\label{eq:volterrag}
\end{align}
such that $p\geq 0$, $Q,S \in \mathbb R^{2\times 2}$,  $q \in \R^2$ and $s^i$ is an $\mathcal F_T$-measurable random variable in $\R^2$, $i=1,\ldots, N$. 

Since the dynamics of each $Z^i$ in \eqref{eq:volterraZi} are linear in $u^i$ and $\bar u$, and $f^{Vol}_i$ and $g_i^{Vol}$ are linear-quadratic in $(Z^i,u^i)$, it is clear that $\mathcal J^i_{Vol}$ is a linear-quadratic functional  in $(u^i, \bar u)$ as in \eqref{eq:J^i}. This is summarized in the following Lemma. 

\begin{lemma} \label{lem:Volterragame}     
The objective functional \eqref{eq:volterraJ} for the stochastic Volterra game can be written in the form of \eqref{eq:J^i} with the following coefficients
\begin{align}
\label{eq:blockA}
\begin{pmatrix}
    \hat A_2(t,s) &  A_3(t,s)   \\
    A_4(t,s) &  A_1(t,s)  
\end{pmatrix} &= 1_{\{s< t\}}D(T,t)^\top (S + S^\top) D(T,s)  \\
& \quad +  1_{\{s< t\}} \int_0^T D(r,t)^\top (Q + Q^\top) D(r,s) dr 
   - \begin{pmatrix}
     q^\top D(t,s)   \\
    0_{\R^2}^\top
\end{pmatrix},  
\end{align}

\begin{align*}
\begin{pmatrix}
    b^i_t   \\
    b^0_t  
\end{pmatrix} &= D(T,t)^\top \left( \E_t [s^i] -  ( S + S^\top)  \E_t [d^i_T]   \right) -   \int_t^T D(r,t)^\top  ( Q + Q^\top)   \E_t [d^i_r]  dr   \\
&   \quad \qquad \quad + \begin{pmatrix}
(d^i_t)^\top q    \\
 0   
\end{pmatrix} , \\
    c^i& = \mathbb   - \int_0^T (d_t^i)^\top   Q d_t^idt -(d_T^i)^\top S  d_T^i   +   (d^i_T)^\top s^i , \quad i=1,\ldots, N,\\
     \lambda &= p. 
    \end{align*}
\end{lemma}

\begin{proof} 
The proof is a straightforward application of Fubini's Theorem and the tower property of conditional expectations.
\end{proof}

\begin{remark}\label{R:Riccati} 
The  solutions that we derived in \eqref{eq:opt_ui} and \eqref{eq:opt_ubar} apply to the particular case of stochastic Volterra LQ games of the form \eqref{eq:volterra1}-\eqref{eq:volterra2}. Solutions to Volterra LQ control problems can be characterised in some cases in terms of solutions to operator Riccati equations and to $L^2$-valued BSDEs; see \cite[Section 6]{abi2022optimal}. In other cases they can be related to infinite-dimensional Riccati equations; see~\cite{abi2021linear, hamaguchi2022linear, wang2022linear} for single-player examples. Our derivation provides {operator solutions} to such Riccati equations for the case of dynamics without control in the volatility. Furthermore, we point out that our expressions share some similarities with formulas that have recently appeared in the computations of Laplace transforms of some quadratic functionals of non-controlled Volterra processes \cite{abi2022characteristic,abi2021markowitz}. 
\end{remark}

In the following we will show that the dynamics \eqref{eq:volterra1} include any stochastic Volterra equation for the state variables $X^i$, where the drift has linear dependence in $X^i$ itself and in $\bar X := \frac  1 N\sum_{j=1}^N X^j$. In order to do that we first introduce the notion of the resolvent of a kernel.

We recall the definition of the product of kernels. For any $G, H \in  L^2\left([0,T]^2,\mathbb R^{}\right)$ we define the $\star$-product as follows
\begin{align*}
(G \star H)(s,u) = \int_0^T G(s,z) H(z,u)dz, \quad  (s,u) \in [0,T]^2,
\end{align*}
which is a well-defined kernel in $L^2\left([0,T]^2,\mathbb R^{}\right)$ due to Cauchy-Schwarz inequality.  Denoting by $\boldsymbol{G}$ and $\boldsymbol{H}$ the two integral operators induced by the kernels $G$ and $H$, we get that $\boldsymbol{G}\boldsymbol{H}$ is  an integral operator induced by the kernel $G\star H$.

For a kernel $K \in L^2([0,T]^2,\R)$, we define its resolvent $R_T \in L^2([0,T]^2,\R)$ by the unique solution to 
\begin{align}\label{eq:resolventeqkernel}
R_T = K + K \star R_T, \quad  \quad  K \star R_T =  R_T \star K.
\end{align} 
Note that the resolvent of a Volterra operator is also a Volterra operator.
In terms of integral operators, this translates into 
\begin{align*}
\boldsymbol{R}_T =  \boldsymbol{K} + \boldsymbol{K}\boldsymbol{R}_T, \quad \boldsymbol{K}\boldsymbol{R}_T=\boldsymbol{R}_T\boldsymbol{K}.
\end{align*}
In particular, if $K$ admits a resolvent,  $({\id}-\boldsymbol{K})$ is invertible and
\begin{align}\label{eq:integralreso}
({\id}-\boldsymbol{K})^{-1}=\id+\boldsymbol{R}_T.
\end{align}

\begin{lemma}\label{L:volterralinear}
For $i=1,\ldots, N$, let  $M^i$ be a progressively measurable process with sample paths in $L^2([0,T],\R)$ and  $K,H:[0,T]^2 \to \mathbb R$ be two Volterra kernels in $\mathcal{G}$. The $N$-dimensional linear system of Volterra equations 
    \begin{align}\label{eq:Xilinear}
        X^i_t  = M^i_t + \int_0^t K(t,s) X^i_s ds + \int_0^t H(t,s) \bar X_s ds, \quad i=1,\ldots, N,
    \end{align}  
    where $\bar X = \frac 1 N \sum_{j=1}^N X^j$,
    admits a unique solution given by 
    \begin{align}\label{eq:Xilinear2}
    X^i_t =  \widetilde M^i_t + \int_0^t R^K(t,s) \widetilde M^i_sds, \quad i=1,\ldots, N,
    \end{align}
    with 
    \begin{align*}
    \widetilde M^i_t &=  M^i_t + \int_0^t H(t,s) \bar M_s ds  + \int_0^t \left( H\star R^{K+H} \right)(t,s) \bar M_s ds, \quad i=1,\ldots, N,  \\
             \bar M_t &= \frac 1 N \sum_{j=1}^N M^j_t,   
    \end{align*}
    and  $R^K,R^{K+H}$ the resolvents of $K$ and $K+H$, respectively.  
\end{lemma}

\begin{proof} We first observe that any Volterra kernel in $\mathcal G$ admits a resolvent kernel, see for instance Corollary 9.3.16 in \cite{gripenberg1990volterra}. 
Thanks to the resolvent equation \eqref{eq:integralreso}, the  solution to the linear Volterra equation satisfied by the mean process $\bar X$,
\begin{align*}
    \bar X_t = \bar M_t  + \int_0^t \left(K(t,s) + H(t,s)\right) \bar X_s ds, 
\end{align*}
is given by 
\begin{align}
    \bar X_t = \bar M_t + \int_0^t R^{K+H}(t,s) \bar M_s ds.
\end{align}
Plugging this expression back in \eqref{eq:Xilinear} together with another application of \eqref{eq:integralreso} yields \eqref{eq:Xilinear2}.
\end{proof}

\begin{remark}\label{R:volterrasimple}
A particular case of interest in our application of Lemma~\ref{L:volterralinear} to controlled Volterra processes is given by
\begin{align*}
   M^i_t =  P^i_t+ \int_0^t  G(t,s)  u^i_s ds + \int_0^t \bar G(t,s) \bar u_s ds,
\end{align*}   
where $P^i$ is progressively measurable, $G$, $\bar G$ are Volterra kernels in $\mathcal G$ and $u$ is the control. 
In particular, we will apply Lemma~\ref{L:volterralinear} with $\bar G = K \equiv 0$ so that, using the fact that 
$H + H\star R^H = R^H$, the solution in \eqref{eq:Xilinear} simplifies to the form of \eqref{eq:volterra1}, that is, 
\begin{align}
    X^i_t &=  M^i_t   + \int_0^t \left( H + H\star R^{H} \right)(t,s) \bar M_s ds\\
    &=  M^i_t   + \int_0^t  R^{H}(t,s) \bar M_s ds \\
    &= P^i_t +  \int_0^t  R^{H}(t,s) \frac 1 N \sum_{j=1}^N 
 P_t^j ds + \int_0^t  (R^{H}\star G)(t,s) \bar u_s ds    +\int_0^t G(t,s) u^i_s ds, \label{eq:Xilinear3}
    \end{align}
for $i=1,\ldots, N.$
\end{remark}

\subsection{Propagator models: optimal liquidation with  transient impact and signals}\label{S:marketimpact} 

The framework for optimal liquidation with transient impact and price predicting signals was first introduced in \cite{Lehalle-Neum18}. The stochastic game, where several financial agents aim to liquidate in an optimal way their positions in a risky asset was presented in \cite{neuman2021trading}. In the following we briefly describe the model. 

\noindent \textbf{The model.} We consider $N$ traders with an initial position of $x^i_0 \in \mathbb{R}$ shares in a risky asset for the $i$-th trader, $i=1,\ldots, N$. The number of shares the trader holds at time $t$ is prescribed as 
    \begin{align} \label{def:Q}
    \wt X^i_t = x^i_0 -\int_0^t u^i_s ds, 
    \end{align}
where $(u^i_s)_{s \in [0,T]}$ denotes the selling speed which is chosen by the trader. We assume that the traders' individual and aggregated trading activity causes price impact on the risky asset's execution price. The actual price at which the orders are executed for the $i$-th trader is given by
\be \label{def:S}
S^i_{t} := N^i_{t} - \lam u^i_t -  \wt Y^{\bar{u}}_t, \qquad 0 \leq t \leq T, 
\ee
where $N^i$ denotes some unaffected progressively measurable price process incorporating an individual (or common) trading signal. The process
\begin{align}\label{z-def} 
\wt Y_t^{\bar u}= \int_0^t G(t,s)\bar u_s ds, \qquad 0 \leq t \leq T,
\end{align}
captures an aggregated linear decaying price impact with a general propagator kernel $G$ and $\lambda>0$ captures instantaneous slippage costs which are incurred by the $i$-th trader's execution strategy.  

\begin{remark}
In \citep{neuman2022optimal, neuman2021trading} an exponentially decaying kernel of the form $G(t,s)= 1_{\{s < t\}} H(t-s)$ with $ H(t) = e^{-\rho t}$ for some $\rho > 0$ is specified, so the control problem can be regarded as Markovian. However, in practice, the decay of the price impact has been shown to be slower than exponential, and more realistically modeled by a power-law kernel, which makes the problem non-Markovian and {difficult} to solve. For the single player case, the problem with general propagator $G \in \mathcal G$ and a progressively measurable $P$ was solved recently in \cite{abi2022optimal}. We refer to \cite[Example 2.5]{abi2022optimal} for some examples of realistic kernels and to the references therein for the empirical motivation.    
\end{remark}

\noindent \textbf{The objective criterion. } The aim of each trader is to maximize the profit and loss from implementing the trading strategy $u^i$ given by 
$$ -\int_0^T S^i_t d\wt X^i_t  + N^i_T \wt X^i_T = \int_0^T S^i_t u^i_t dt + N^i_T \wt X^i_T$$
in the presence of running and terminal penalizations on the trader's inventory of the form $\phi (\wt X^i_t)^2$  and $\varrho (\wt X^i_T)^2$ with $\phi,\varrho>0$, which encourage the trader to liquidate her position. This leads to the following performance functional 
\begin{equation} \label{def:objective}
J^{i}_{liq}(u^i) = \mathbb{E} \Bigg[ \int_0^T   f^{liq}_{i}(t,\wt X^i_t,\wt Y^{\bar u}_t,u^i_t) dt   + g^{liq}_i (\wt X^i_T)\Bigg],
\end{equation}
with 
\begin{align}
    f^{liq}_{i}(t,x,y,u) &=   -\lambda u^2 - \phi x^2   - u y +   N^i_t u , \quad   
    g^{liq}_i (x) = - \varrho x^2 + N^i_T x.  
\end{align}

\noindent \textbf{The correspondence with our Volterra game framework.}
\begin{tcolorbox}
The optimal liquidation problem with aggregated transient price impact and signals corresponds to a Volterra game in the sense of Section~\ref{S:Volterragame} with  controlled variables $X^i$ in~\eqref{eq:volterra1} and $Y^i$ in~\eqref{eq:volterra2} given by $(X^i, Y^i)  = (\wt X^i, N^i- \wt Y^{\bar u})$ with
$$  G_2(t,s) = - 1_{s\leq t}, \quad    G_3 = 0 \quad G_1 = -G  \quad P^i = x_0^i, \quad R^i = N^i.$$
The objective criterion in~\eqref{def:objective} can be written as in~\eqref{eq:volterraJ} in terms of $f^{Vol}$ in~\eqref{eq:volterraf} and $g^{Vol}$ in~\eqref{eq:volterrag} where 
\begin{align*}
    p &= \lambda ,  \quad 
    Q =  {\begin{pmatrix} \phi & 0 \\ 0 & 0 \end{pmatrix}},  \quad 
    q =  {(0, 1)^\top},  \quad 
    S =  {\begin{pmatrix} \varrho & 0 \\ 0 & 0 \end{pmatrix}}, \quad 
    s^i   = (P^i_T,0)^\top. 
\end{align*}
\end{tcolorbox}

In this case, using \eqref{eq:blockA}, the kernels $(A_1, \hat A_2, A_3, A_4)$ take the form  
$$
A_1(t,s) = A_4(t,s) = 0, \; 
\hat A_2(t,s) = 2 \varrho\, 1_{\{s<t\}} + 2 \phi (T-t)\, 1_{\{s<t\}}, \; 
A_3(t,s) = 1_{\{s<t\}}\, G(t,s),
$$
so that our main Assumption~\ref{assum-op} holds whenever the kernel $G$ is nonnegative definite; see \cite[Lemma~5.4]{abijaber2024optimal} for obtaining the representation for $\hat A_2$.

\begin{remark}[Improvements with respect to the related literature] \label{rem-imp3}
In the single player case, the above non-Markovian problem was solved only very recently in~\cite{abi2022optimal} by making an ansatz on the value function combined with a martingale verification argument. Our approach not only extends such results of \cite{abi2022optimal} to a multi-player price impact game but also provides a new and direct argument for deriving the solution for the single player case. 
\end{remark}

\subsection{Delay in the control: an inter-bank lending and borrowing model}\label{S:exsystematic}

In this subsection we show how our results extend models on systemic risk with delay studied in \cite{carmona2018systemic,fouque2018mean}.
 
\noindent \textbf{The model.}  The log-monetary reserves $x^i$ of $N$ banks, $i=1,\ldots, N$, are   modeled as follows
\begin{align}\label{eq:systemicvar}
dx^i_t = \left( h_i(t) + \int_{[0,t]} \nu(ds) u^i_{t-s} \right) dt + \sigma_i dW^i_t ,\quad x^i_0 \in \R, 
\end{align}
with $h_i$ a progressively measurable process, $(W^1,\ldots, W^N)$ an $N$-dimensional Brownian motion and  $\nu$ a (signed) measure on $[0,T]$ of locally bounded variation. The control $u^i$ of the $i$-th bank corresponds to the rate of lending or borrowing (depending on the sign of $u^i$) from a central bank. The delay in the control reflects the repayments after a fixed time. The main example is given by $u^i_t - u^i_{t-\tau }$ which corresponds to the case $\nu = \delta_0 - \delta_{\tau}$ for some fixed time $\tau \geq 0$. In this case, if the $i$-th bank borrowed from the central bank an amount $u_{t}dt$ at some time $t$, then $x^i_t$ increases by $u_t dt$. In addition, since the amount needs  to be paid back to the central bank at a later time $t + \tau $, the log-monetary reserve  $x^i_{t+\tau}$  decreases by $u^i_{(t+\tau) - \tau} dt=u^i_t dt$, which explains the form of the drift $u^i_t - u^i_{t-\tau}$ at time $t$.

\begin{remark}
    We note that in \cite{carmona2018systemic,fouque2018mean}, the dynamics of $x^i$ are stated slightly differently using the convention of delayed differential equations: 
    \begin{align*}
   d x^i_t = \left( \int_{[0,\tau]} \theta(ds) u^i_{t-s} \right) dt + \sigma_i dW^i_t, \quad x^i_0 \in \mathbb R, 
    \end{align*}
for some fixed $\tau \geq 0$,    with an initialization of the control of the form 
    $$ \quad u^i_s = \phi^i(s), \quad \mbox{for } s \in [-\tau, 0],$$
    for some specified function $\phi$ and measure $\theta$. 
    Setting $\nu(ds) = 1_{s\leq \tau} \theta(ds)  $  and $h_i(t) =  1_{t\leq \tau }\int_{[t,\tau]} \theta(ds) u^i_{t-s} = 1_{t\leq \tau } \int_{[t,\tau]}\theta(ds) \phi^i(t-s) $ yields the dynamics in~\eqref{eq:systemicvar}.
\end{remark}

\noindent \textbf{The objective criterion.} The $i$-th bank chooses the strategy $u^i$ to minimize the following objective criterion that involves the aggregated average $\bar x = \frac{1}{N} \sum_{i=1}^N x^i$:
\begin{align} \label{eq:objective_sys}
    \mathcal J^i_{sys}(u^i, u^{-i}) =\mathbb E\left[ \int_0^T f^{sys}(x^i_t, \bar x_t , u^i_t) dt + g^{sys}(x^i_T, \bar x_T) \right],
\end{align}
with the running and terminal costs given by
\begin{align*}
f^{sys}(x, \bar x , u) &= \frac{u^2}2 - \beta u(\bar x - x ) + \frac{\varepsilon}{2}(\bar x -x )^2, \\
g^{sys}(x, \bar x) &= \frac{c}{2}(\bar x - x)^2.
\end{align*}
The running cost of borrowing/lending is given by $(u^i)^2/2$, the parameter $\beta>0$ controls the incentive to borrow or lend depending on the difference with the average capitalization level; the quadratic term in $(\bar x - x)^2$ in the running and terminal costs penalize departure from the average. The condition $\beta^2\leq \epsilon$ ensures convexity.

\noindent \textbf{The correspondence with our Volterra game framework.} Setting 
$$G(t):= \nu([0,t]), \quad t\geq 0,$$ writing \eqref{eq:systemicvar} in integral form, and applying Lemma~\ref{L:Fubini} below,   we get the following Volterra representation for $x^i$:
\begin{align}\label{eq:sysrep}
   x_t^i =  P^i_t +  \int_0^t G(t-s) u^i_s ds 
\end{align}
with $P^i = x^i_0 + \int_0^{\cdot} h_i(s)ds + \sigma_i W^i$.

\begin{tcolorbox}
Clearly, this corresponds to a  Volterra game in the sense of Section~\ref{S:Volterragame} where the controlled variables $X^i$  in \eqref{eq:volterra1} and $Y^i$ in \eqref{eq:volterra2} are given by $(X^i, Y^i)  = (x^i,\bar x)$ with  
$$  G_1 = {G_2} = G, \quad {G_3 = 0}, \quad R^i = \frac{1}{N} \sum_{j=1}^N P^j,$$
and where the negative of the objective criterion in \eqref{eq:objective_sys} to be maximized is given in terms of $f^{Vol}$ as in \eqref{eq:volterraf} and $g^{Vol}$ as in \eqref{eq:volterrag} with parameters
\begin{align*}
    p &= \frac 1 2,  \quad 
    Q =  \frac{\epsilon}{2} \begin{pmatrix}
        1  & -1  \\
        -1   & 1
     \end{pmatrix},  \quad 
    q = \beta(-1, 1)^\top,  \quad 
    S =   \frac{c}{\epsilon} Q, \quad 
    s^i   = (0,0)^\top. 
\end{align*}
\end{tcolorbox}

In this case, using \eqref{eq:blockA}, the kernels $(A_1, \hat A_2, A_3, A_4)$ take the form 
\[
A_1(t,s) = \hat A_2(t,s) - 1_{s < t} \beta G(t,s) , \quad 
A_3(t,s) = -\hat A_2(t,s), \quad A_4(t,s) = -A_1(t,s)
\]
and
\[
\hat A_2(t,s) = 1_{\{s<t\}} \left( c\, G(T,t)G(T,s) + \epsilon \int_0^T G(r,t)G(r,s)dr + \beta G(t,s) \right) .
\]
Hence, our main Assumption~\ref{assum-op} is in general not satisfied. However, Assumption~\ref{assum-op} is not necessary for solving the game, and our explicit solutions in Theorem~\ref{thm:main-finite} provide the first-order condition in closed form whenever the game is well-posed, which can in certain specific cases be obtained under weaker assumptions than Assumption~\ref{assum-op}.

\begin{example}
We briefly point out some {examples} for $G$ in this setting. 
\begin{itemize}
     \item [\textbf(i)] If $\nu(ds)=g(s) ds$ for some locally integrable function $g$, then 
     $G(t) = \int_0^t g(s) ds$, for all $t\geq 0$.
     \item  [\textbf(ii)] If $\nu(ds) =  \delta_0 - \delta_{\tau}(ds)$ for some $\tau \geq 0$, then $G(t) = 1_{\{t \geq 0\}} -  1_{\{t \geq \tau \}}$. 
     \end{itemize}
\end{example}

\begin{remark}[Improvements with respect to the related literature] \label{rem-imp1} In \cite{carmona2018systemic}, the equilibrium to the above inter-bank lending and borrowing game was characterised as a solution to a system of integral equations  \cite[equations (41)-(45)]{carmona2018systemic}. Our main results not only provide  novel {semi-explicit} operator formulas for the game, but also allow with no additional effort,  non-trivial realistic extensions of the game.
First, we allow the inclusion of  lending and
borrowing between the banks by introducing an interacting term
$$ \frac{1}{N}\sum_{j=1}^N \left(x^j_t - x^i_t\right), $$
in the  drift of $x^i$ in \eqref{eq:systemicvar}. This term represents the
rate at which bank $i$ borrows from or lends to bank $j$ in the spirit of the model in \cite{carmona2013mean} but with the addition of the delay feature, which was not considered \cite{carmona2018systemic}. 
Second, we can include a common Brownian noise $W$ by updating $P^i$ to $P^i = \int_0^{\cdot} h_i(s) ds + \sigma^i W^i + \tilde \sigma^i W$ in \eqref{eq:sysrep}.  Third, we allow the possibility of adding a more general independent and  correlated noise which is not necessarily a semimartingale, where the usual techniques that rely on Itô's formula  break down. Finally, we allow the inclusion of delays in the mean-field interacting term in the drift, such feature is discussed in the next section.
\end{remark}

\begin{lemma}\label{L:Fubini}
Let $\nu$ be a measure on $\R_+$ of locally finite total variation.  Set $G(t)=\nu([0,t])$, for all $t\geq 0$,  then, for any locally  integrable measurable function $u$ it holds that 
    \begin{align*}
\int_{[0,t]} G(t-s) u_s ds = \int_{[0,t]} \left(\int_{[0,l]}  u_{l-r}  \nu(dr) \right) dl, \quad t \geq 0. 
\end{align*}
\end{lemma}

\begin{proof} 
This follows from  an application of Fubini's theorem. 
\end{proof}
\subsection{Delay in the state:  an advertising model with mean-field effect}\label{S:delaystate}
 
Our next application is motivated by problems of stochastic optimal advertising \cite{nerlove1962optimal, gozzi245stochastic, bensoussan2017linear} with time-delay and persistence of advertising expenditures, in the presence of competition between  products or firms. 

\noindent \textbf{The model.} We consider a company launching $N$ products such that the dynamics for the  advertising goodwill  $x^i$ of the $i$-th product to be launched are given by
\begin{align}\label{eq:stochastic integro}
dx^i_t = \left( \beta u^i_t + h_i(t) + \int_{[0,t]} \nu(ds) x^i_{t-s} +  \int_{[0,t]} \mu(ds) \bar x_{t-s}   \right) dt + \sigma_i dW^i_t,  \quad  x^i_0 \in \mathbb R. \qquad
\end{align}
Here $u^i$ is the intensity of advertising spending, $(W^1,\ldots, W^N)$ is a possibly correlated $N$-dimensional Brownian motion, $\beta \geq 0$ is a constant
advertising effectiveness factor. $h_i$ is a deterministic or external stochastic factor for image deterioration in the absence of advertising which can be common to several products. $\nu$ is the distribution of the forgetting time, and $\mu$ is the distribution of a time-delayed interactive competition among the products.  Note that when
$\nu$, $h_i$, $\mu$, and $\sigma_i$ are identically zero, equation \eqref{eq:stochastic integro} reduces to the classical model of \citet{nerlove1962optimal}.

\begin{remark}
    We note that  the dynamics of $x^i$ are usually stated slightly differently using the convention of delayed differential equations: 
    \begin{align*}
   d x^i_t = \left( \beta u^i_t + \int_{[0,\tau]} \theta_1(ds) x^i_{t-s} +  \int_{[0,\tau]} \theta_2(ds) \bar x_{t-s}   \right) dt + \sigma_i dW^i_t, 
    \end{align*}
    for some fixed $\tau \geq 0$,    with an initialization of the controlled variables of the form 
    $$ \quad x^i_s = \phi^i(s), \quad \mbox{for } s \in [-\tau, 0],$$
    for some specified functions $\phi^i$ and measures $\theta_1$ and $\theta_2$. 
    Setting $\nu(ds) = 1_{s\leq \tau} \theta_1(ds)  $,  $\mu(ds)  = 1_{s\leq \tau} \theta_2(ds) $
    and 
    \begin{align}
    h_i(t) &=  1_{t\leq \tau }\int_{[t,\tau]} \left(\theta_1(ds) x^i_{t-s} +  \theta_2(ds) \bar x_{t-s}\right)\\
    &= 1_{t\leq \tau } \int_{[t,\tau]} \left(\theta_1(ds) \phi^i(t-s) +  \theta_2(ds)\frac 1 
 N \sum_{j=1}^N \phi^j(t-s)\right),
    \end{align}
 yields the dynamics \eqref{eq:stochastic integro}.
\end{remark}

\noindent \textbf{The objective criterion.}
The following objective functional to be maximized was presented in \cite[Section 6]{gozzi245stochastic}, 
\begin{align}\label{eq:objectivead}
    \mathcal J^i_{adv}(u^i, u^{-i}) =\mathbb E\left[ - \int_0^T \lambda (u^i_t)^2   dt +  \beta x^i_T  \right].
\end{align}

\noindent \textbf{The correspondence with our Volterra game framework.}
We start by showing that the dynamics \eqref{eq:stochastic integro} can be written in the form \eqref{eq:volterra1}. 
First, an application of Lemma~\ref{L:Fubini} shows that  \eqref{eq:stochastic integro} can be written in  the Volterra form 
\begin{align}\label{eq:stochasticintegro}
x^i_t = M^i_t  +  \int_0^t K(t-s)x^i_{s}ds +  \int_0^t H(t-s) \bar x_{s}ds, 
\end{align}
with $M^i=\int_0^{\cdot}h_i(s) ds + \sigma_i W^i  +  \int_0^{\cdot} \beta u^i_s ds  $, $K(t)=\nu([0,t])$ and $H(t)=\mu([0,t])$. 
Then, an application of Lemma~\ref{L:volterralinear} shows clearly that $x^i$  can be written in the form \eqref{eq:volterra1}. {However, to ease notations, we will highlight the application of Lemma~\ref{L:volterralinear} for the case {$H$ locally integrable,} $\nu =0$ and $h_i=0$, so that $K=0$:}
\begin{tcolorbox}
   Thanks to Remark~\ref{R:volterrasimple} applied to \eqref{eq:stochasticintegro} for the case $\nu=0$ and $h_i=0$,  we can deduce that the advertising problem with delay in the  mean-field interaction corresponds  to a  Volterra game where the controlled variables $X^i$  in \eqref{eq:volterra1} and $Y^i$ in \eqref{eq:volterra2} are given by $(X^i, Y^i)  = (x^i,0)$ with
$$  G_2(t,s) =  \beta 1_{s\leq t} , \quad G_3(t,s) = \beta \int_s^{t}R^H(t,u) du, \quad G_1 = 0,$$
$$ P^i = \sigma_i W^i + \int_0^{\cdot} R^H(\cdot,s) \frac 1 N \sum_{j=1}^N \sigma_j W^j_s ds, \quad R^i =0, $$
where $R^H$ is the resolvent of $H$,
and where the objective criterion in \eqref{eq:objectivead} is defined in terms of $f^{Vol}$ of \eqref{eq:volterraf} and $g^{Vol}$ of \eqref{eq:volterrag} with 
\begin{align*}
    p &= \lambda ,  \quad 
    Q = 0_{\R^{2\times 2}},  \quad 
    q = (0, 0)^\top,  \quad 
    S =  0_{\R^{2\times 2}} , \quad 
    s^i   = (\beta,0)^\top. 
\end{align*}
Note that since $H$ is locally integrable, the existence of the resolvent $R^H$ is ensured by \cite[Theorem 2.3.1]{gripenberg1990volterra}.
\end{tcolorbox}

In this case, using \eqref{eq:blockA}, the kernels $(A_1, \hat A_2, A_3, A_4)$ are all equal to $0$
so that Assumption~\ref{assum-op} is trivially satisfied.

\begin{remark}[Improvements with respect to the related literature]\label{rem-imp2}  To the best of our knowledge, the advertising multi-player game above has not been solved yet. A similar advertising game has been briefly presented  in \cite[Section 2.2]{bensoussan2017linear} in the context of certain Linear-Quadratic Stackelberg games.  In \cite{gozzi2015stochastic}, a one-player game has been considered with $\mu$ and $\nu$ in \eqref{eq:stochastic integro} absolutely continuous with respect to the Lebesgue measure. Such control problem has been characterized, under the particular setting of \eqref{eq:objectivead}, in terms of a system of integral equations, see \cite[equations (30)--(32)]{gozzi2015stochastic}. Our framework  provides the first {solution method} for a multi-player advertising game with delay in the individual state and in the mean-field state. {Note that it} also allows for the inclusion of delay in the control, common and non-semimartingale noise which are usually not straightforward extensions. As the authors point out in \cite[Section 5]{gozzi2015stochastic}: the only technique that seems to work when combining delay in the control and in the state is the framework of viscosity solutions which does not provide enough regularity. 
\end{remark}

\begin{remark} Combined with the ideas in Section \ref{S:exsystematic}, we can also incorporate linear time delays in the control $u^i$ and in the dynamics of $x^i$ in  \eqref{eq:stochastic integro}; see Lemma~\ref{L:volterralinear} and Remark~\ref{R:volterrasimple}. Our framework allows for {semi-explicit} solutions even when the delays depend on a measure, which usually introduces technical difficulties with the usual infinite dimensional lifts applied in the study of controlled delayed equations; see, for instance,  \cite[Remark 1]{gozzi245stochastic} and \cite{gozzi2015stochastic}.
\end{remark}

\section{The Mean-Field Game} \label{sec:MeanFieldGame}
In this section we formulate and solve the corresponding asymptotic mean-field game version of the finite-player game from Section~\ref{sec:FinitePlayer} when the number of players $N$ tends to infinity. We prove that the finite-player Nash equilibrium strategies converge in the infinite population limit to the mean-field game equilibrium and show that the latter yields an approximate equilibrium for the finite-player setup from Section~\ref{sec:FinitePlayer}. We first consider a generic player in the mean-field game limit and then provide an equivalent formulation in terms of infinitely many players. The main advantage of the analysis in this section is that the mean-field limit provides both simpler expressions to the equilibrium strategies and more relaxed conditions for existence of an equilibrium, compared to the finite-player game. 

Let $(\Omega, \mathcal F_T, \mathbb F=(\mathcal F_t)_{0 \leq t\leq T}, \mathbb P)$ denote the filtered probability space from Section~\ref{sec:FinitePlayer}. Recall that we use the notation $\mathbb E_t = \mathbb E [ \cdot | \mathcal F_t]$ to represent the conditional expectation with respect to $\mathcal{F}_t$. Moreover, let $c$ denote an $\mathcal F_T$-measurable random variable and let $\beta, \beta^0$ denote $\mathbb F$-progressively measurable processes such that $\beta^0$ is independent of~$\beta$. We set
\begin{equation} \label{def:b_generic}
     b := \beta + \beta^0
\end{equation}
and denote by $\mathbb F^0:=(\mathcal F^0_t)_{0 \leq t\leq T}$ the filtration generated by $\beta^0$ satisfying the usual conditions. 

\begin{assumption} \label{ass:filtration}
    We assume that for all $t \in [0,T]$, $\mathcal{F}_T^0$ and $\mathcal{F}_t$ are conditionally independent given $\mathcal{F}^0_t$.
\end{assumption}

\begin{remark} \label{rem-filt} 
    For instance, if we denote by $(\mathcal F^{\beta}_t)_{t\leq T}$ the filtration generated by the process $\beta$, then the filtration $\mathbb F$ defined by $\mathcal F_t := \mathcal F^{\beta}_t \vee \mathcal F^0_t $ for all $t\leq T$, satisfies Assumption~\ref{ass:filtration}. 
\end{remark}

\begin{remark} 
We illustrate Assumption~\ref{ass:filtration} for the example of the portfolio liquidation game in Section~\ref{S:marketimpact}. Suppose that for each agent $i \in \mathbb{N}$ the initial inventory is given by $x_0 \in \mathbb{R}$ and their trading signal is given by $N^i = B^i + S$ with idiosyncratic independent Brownian motions $(B^i)_{i \geq 1}$ which are also independent of a common noise semimartingale $S$. In the mean-field setup, $(B^i)_{i \geq 1}$ changes to a Brownian motion $\wt B$ independent of $S$. With this configuration, we have $\beta^0_t = \mathbb E_t \left[S_t - S_T\right] $, and since $\mathbb E_t [\wt B_t-\wt B_T  ] = 0$, $\beta_t = 2(\varrho + \phi (T-t)) x_0$, which is a special case of Remark \ref{rem-filt}. In particular, we have $b_t =   \mathbb E_t \left[S_t - S_T\right] +2(\varrho + \phi (T-t)) x_0$ for all $t \in [0,T]$.
\end{remark}

In addition to the set $\mathcal{U}$ introduced in~\eqref{def:admissset} we also define the subset of processes
\begin{equation} \label{def:admis_generic}
 \mathcal U^0:=\left\{ v: \,  \mathbb{F}^0 \mbox{-progressively measurable s.t. } \int_0^T \mathbb E[v_s^2] ds <\infty \right\} \subset \mathcal{U}.
\end{equation}
    
The paradigm of the mean-field game limit of the finite-player game from Section~\ref{sec:FinitePlayer} is the following: First, we consider a generic player who seeks to implement a strategy $v \in \mathcal U$ in order to maximize for a fixed $\mu \in \mathcal U^0$ the objective function
\begin{equation}\label{eq:J_generic}
   \begin{aligned} 
    J(v; \mu) := & \, \E\left[ -\langle \mu,\boldsymbol A_1  \mu \rangle_{L^2} - \langle v,\boldsymbol A_2 v \rangle_{L^2} - \langle v, (\boldsymbol{A}_3 + \boldsymbol{A}^*_4)  \mu \rangle_{L^2} \right. \\
    & \quad \left. + \langle b, v \rangle_{L^2} + \langle b^0, \mu \rangle_{L^2} + c\right],
\end{aligned}
\end{equation}  
which is akin to~\eqref{eq:J^i} with same operators $\boldsymbol A_1, \boldsymbol A_2, \boldsymbol A_3, \boldsymbol{A}_4$, as well as $\mathbb{F}$-progressively measurable process $b^0$ as introduced above in Section~\ref{sec:FinitePlayer}. Then, we determine the $\mathbb F^0$-measurable process $\mu$ such that a mean-field game equilibrium with common noise $\beta^0$ is obtained in the following sense. 
  
\begin{definition} \label{def:meanfieldgame}
A pair $(\hat v, \hat \mu) \in \mathcal U \times \mathcal U^{0}$ is called a mean-field game equilibrium if the control $\hat{v}$ solves the optimization problem 
\begin{equation} \label{eq:meanfield_opt}
{\underset{v \in \mathcal U}{\text{maximize}} \; J(v ;\hat{\mu})}
\end{equation}
under the consistency condition 
\begin{align}\label{eq:meanfield_cons} 
\E[ \hat{v}_t | \mathcal F^{0}_T ] = \hat{\mu}_t, \quad \Omega \times [0,T] \text{ almost everywhere.}
\end{align}
\end{definition}

\begin{remark}
Note that due to Assumption~\ref{ass:filtration}, it holds in~\eqref{eq:meanfield_cons} that
\begin{align*}
\E[ \hat{v}_t  | \mathcal F^{0}_T ] = \E [\hat{v}_t | \mathcal F^{0}_t ]
\end{align*}
for all $t \in [0,T]$. Indeed, for all $\xi \in \mathcal{F}^0_T$ it follows that
\begin{equation*}
    \mathbb{E}[\xi \, \E[ \hat{v}_t | \mathcal F^{0}_t ]] = \mathbb{E}[ \E[\xi | \mathcal F^{0}_t ] \, \E[ \hat{v}_t | \mathcal F^{0}_t] ] = \mathbb{E}[ \mathbb{E}[\xi \hat{v}_t \vert \mathcal{F}_t^0]] = \mathbb{E}[\xi \hat{v}_t ],
\end{equation*}
where the second equality is obtained from the conditional independence given by Assumption~\ref{ass:filtration}.
\end{remark}

To state the mean-field game equilibrium, it is convenient to introduce the following two solution maps $F$ and $G$ of two associated Fredholm equations:
\begin{equation} \label{eq:solmapFG_mfg} 
\begin{aligned}
F(t, x) := & \left( (\id - \boldsymbol{\tilde{B}})^{-1} \tilde{a}(x) \right) (t) \\
G(t, x) := & \left( (\id - \boldsymbol{\hat{B}})^{-1} \hat{a}(x) \right) (t) 
\end{aligned}
\qquad \qquad (0 \leq t \leq T),
\end{equation}
where 
\begin{equation} \label{eq:solmapFG_mfg_atilde}
\begin{aligned} 
\tilde{a}_t(x) & := \frac 1{2\lambda} \left(x_t - \langle  1_{\{t \leq \cdot\}} \hat{A}_2(\cdot, t), \boldsymbol{\tilde{D}}_t^{-1} 1_{\{t \leq \cdot\}} \E_t[x_{\cdot}]  \rangle_{L^2}  \right), \\
\tilde{B}(t,s) & :=  1_{\{s \leq t\}} \frac 1{2\lambda} \left( \langle 1_{\{t \leq \cdot\}} \hat{A}_2(\cdot,t) ,  \boldsymbol{\tilde{D}}_t^{-1} 1_{\{t \leq \cdot\}}  \hat{A}_2(\cdot,s) \rangle_{L^2} - \hat{A}_2(t,s) \right), \\
\boldsymbol{\tilde{D}}_t & := 2\lambda \id + (\boldsymbol{\hat{A}}_2)_t  + (\boldsymbol{\hat{A}}^*_2)_t; 
\end{aligned}
\end{equation}
and 
\begin{equation} \label{eq:solmapFG_mfg_ahat}
\begin{aligned} 
\hat{a}_t(x) & := \frac 1{2\lambda} \left(x_t - \langle  1_{\{t \leq \cdot\}} ( \hat{A}_2(\cdot,t) + A_4(\cdot,t) ), \boldsymbol{\hat{D}}_t^{-1} 1_{\{t \leq \cdot\}} \E_t[x_{\cdot}]  \rangle_{L^2}  \right), \\
\hat{B}(t,s) := & \, 1_{\{s \leq t\}} \frac 1{2\lambda} \Big( \langle 1_{\{t \leq \cdot\}} ( \hat{A}_2(\cdot,t) + A_4(\cdot,t)) ,  \boldsymbol{\hat{D}}_t^{-1}  1_{\{t \leq \cdot\}}   (\hat{A}_2(\cdot,s)  + {A}_3(\cdot,s))\rangle_{L^2} \Big. \\
& \hspace{54pt} \Big. - (\hat {A}_2(t,s)  + {A}_3(t,s)) \Big),\\
\boldsymbol{\hat{D}}_t & := 2\lambda \id +  (\boldsymbol{\hat{A}}_2)_t  + (\boldsymbol{\hat{A}}^*_2)_t + (\boldsymbol{A}_3)_t + (\boldsymbol{A}^*_4)_t.
\end{aligned}
\end{equation}
We are now ready to provide the solution to the mean-field game:  

\begin{theorem} \label{thm:meanfieldgame}
Assume that \eqref{ass:P} as well as Assumptions~\ref{assum-op} and~\ref{ass:filtration} are satisfied.
Then, the unique mean-field game equilibrium $(\hat v,\hat \mu) \in \mathcal U \times \mathcal U^{0}$ in the sense of Definition~\ref{def:meanfieldgame} is given by
\begin{align}
 \hat v_t & = F( t,    b -  \boldsymbol{A}_3(\hat \mu) - \boldsymbol{A}_4^*( \mathbb{E}_\cdot \hat \mu)), \label{eq:mfgv} \\ 
 \hat{\mu}_t &= G(t, \mathbb{E}[\beta] + \beta^0) \label{eq:mfgnu}, 
 \end{align}
for all $t \in [0,T]$, where $F$ and $G$ are defined in~\eqref{eq:solmapFG_mfg}. 
\end{theorem}

\begin{remark}
To be more explicit, note that the process which appears on the right-hand side of \eqref{eq:mfgv} is 
\begin{equation*}
    b_s -  \boldsymbol{A}_3(\hat \mu)(s) - \boldsymbol{A}_4^*( \mathbb{E}_s \hat \mu)(s) =  b_s -\int_0^s A_3(s,r)\hat \mu_r dr - \int_s^T A_4(r,s)\mathbb{E}_s [\hat \mu_r] dr , \quad s \leq T. 
\end{equation*}
\end{remark}

The proof of Theorem~\ref{thm:meanfieldgame} is given in Section~\ref{sec-proof-infinite}. We note that Remarks~\ref{R:num} and~\ref{R:Riccati} remain valid for the expressions~\eqref{eq:mfgv} and~\eqref{eq:mfgnu}.

\subsection{The Infinite-Player Game Formulation} \label{subsec:infinteplayer}

It is also very natural to formulate the mean-field game {limit} of the finite-player game introduced in Section~\ref{sec:FinitePlayer} as an infinite-player game when the number of players~$N$ is sent to infinity. We show below that the resulting infinite-player mean-field game Nash equilibrium is equivalent to a generic player's mean-field game equilibrium derived above.
{While there is some overlap between the result of Theorem~\ref{thm:meanfieldgame} and some of the results in this section, we have decided to include the infinite-player game formulation, as it is a more natural framework to study convergence of $N$-player games as $N\rr \infty$, which is done in Section~\ref{subsec:convergence}. Moreover, this approach is often used in applications, particularly in price impact games, where the agents are not necessarily identical (see, e.g., \cite{CasgrainJaimungal:20, DrapeauLuoSchiedXiong:19, fu2020, FuGraeweHorstPopier:20,neuman2021trading}). }

Specifically, consider a countable collection of controls $(v^i)_{i \in \mathbb{N}} \subset \mathcal{U}$ associated with infinitely many agents indexed by $i \in \mathbb{N}$. Moreover, we let $\nu \in \mathcal{U}$ denote a stochastic process which represents the limit of the control processes' averages. Then, the infinite-player game version of each player $i$'s objective from~\eqref{eq:J^i} is given by
\be\label{eq:J_inf}
\begin{aligned}
    J^{i,\infty}(v^i; \nu) = & \, \E\left[ -\langle \nu,\boldsymbol A_1  \nu \rangle_{L^2} - \langle v^i,\boldsymbol A_2 v^i \rangle_{L^2} - \langle v^i, (\boldsymbol{A}_3 + \boldsymbol{A}^*_4) \nu \rangle_{L^2} \right. \\
    & \quad \left. + \langle b^i, v^i \rangle_{L^2} + \langle b^0, \nu \rangle_{L^2} + c^i \right],
\end{aligned}
\ee
with same operators $\boldsymbol A_1, \boldsymbol A_2, \boldsymbol A_3, \boldsymbol{A}_4$, as well as $\mathbb{F}$-progressively measurable processes $b^0, (b^i)_{i \in \mathbb{N}}$ and $\mathcal{F}_T$-measurable $(c^i)_{i \in \mathbb{N}}$ as introduced above in Section~\ref{sec:FinitePlayer}.

The notion of a Nash equilibrium in this infinite-player mean-field game is to first solve for a fixed $\nu \in \mathcal{U}$ simultaneously for each player $i \in \mathbb{N}$ the optimal stochastic control problem        

\begin{equation}
{\underset{v^{i} \in \mathcal{U}}{\text{maximize}} \; J^{i,\infty}(v^{i};\nu) } \qquad (i \in \mathbb{N}),
\end{equation}
and then to determine the process $\nu$ in equilibrium in the following sense.

\begin{definition} \label{def:Nash_Inf}
A collection of processes $( (\hat{v}^i)_{i \in \mathbb{N}},\hat{\nu}) \subset \mathcal{U}$ is called a mean-field game Nash equilibrium with mean-field strategy $\hat{\nu}$ if for all $i \in \mathbb{N}$ the control $\hat{v}^i$ solves the optimization problem 
\begin{equation} \label{def:Nash_Inf_opt}
{\underset{v^{i} \in \mathcal{U}}{\text{maximize}} \; J^{i,\infty}(v^{i};\hat{\nu}) } \qquad (i \in \mathbb{N})
\end{equation}
and $\hat{\nu}$ satisfies the consistency condition
\begin{equation} \label{def:consistency_MFG}
\lim_{N \rightarrow \infty} \sup_{0\leq t\leq T}\E \left[ \left(\frac{1}{N} \sum_{i=1}^N  \hat{v}^i_t - \hat{\nu}_t \right)^2\right] = 0.
\end{equation}
\end{definition}

We make the following assumption on the players' individual processes $(b^i)_{i \in \mathbb{N}}$ in~\eqref{eq:J_inf}:

\begin{assumption} \label{ass:inf_Player} 
We assume that there exists a unique $\mathbb{F}$-progressively measurable process $b^{\infty} \in L^2(\Omega \times [0,T])$ and a bounded function $h:\re_+ \mapsto \re_+$ with $h(x) \rr 0$ as $x \rr \infty$ such that
\begin{equation} \label{ass:infPlayer_b:eq}
  \sup_{0\leq t\leq T}\E \left[ \left(\frac{1}{N} \sum_{i=1}^N  b^i_t -b^{\infty}_t \right)^2\right] \leq  h(N), \quad \textrm{for all } N \geq 1. 
\end{equation}
\end{assumption}

\begin{remark} 
The general Assumption~\ref{ass:inf_Player} can be made precise in specific examples. For instance, consider the portfolio liquidation price impact game from Section~\ref{S:marketimpact} above. Suppose that for each agent $i \in \mathbb{N}$ the initial inventory is given by $x^i_0 \in \mathbb{R}$ and their trading signal is given by $N^i = B^i + B$ with idiosyncratic independent Brownian motions $(B^i)_{i \geq 1}$ which are also independent of a common noise Brownian motion $B$. With this configuration, we have 
\begin{equation}
b^i_t = 2(\varrho + \phi (T-t)) x^i_0 + \mathbb E_t\left[B^i_t-B^i_T \right] + \mathbb E_t \left[B_t - B_T\right] =   2(\varrho + \phi (T-t)) x^i_0, \qquad t \leq T.
\end{equation}
Therefore, the condition in~\eqref{ass:infPlayer_b:eq} boils down to requiring that the limit $x^\infty \in \mathbb{R}$ of the averages of all initial positions $(x^i_0)_{i \in \mathbb{N}}$ exists, i.e.,  
\begin{equation}
x^\infty = \lim_{N \rightarrow \infty} \frac{1}{N} \sum_{i=1}^N x^i_0.
\end{equation}
Moreover, the process $b^\infty$ in~\eqref{ass:infPlayer_b:eq} is then given by 
\begin{equation}
b^\infty_t = 2(\varrho + \phi (T-t)) x^\infty,  \qquad  t \leq T.
\end{equation} 
\end{remark}

Akin to Theorem~\ref{thm:meanfieldgame}, the solution to the infinite-player mean-field game is characterized as follows:

\begin{theorem} \label{thm:meanfieldgame_inf}
Assume that \eqref{ass:P} as well as Assumptions~\ref{assum-op} and~\ref{ass:inf_Player} are satisfied.
Then, the unique solution $((\hat{v}^i)_{i \in \mathbb{N}},\hat{\nu}) \subset \mathcal{U}$ of the infinite-player mean-field game in the sense of Definition~\ref{def:Nash_Inf} is given by
\begin{align}
 \hat v^i_t &= F(t,  b^i - \boldsymbol{A}_3 (\hat \nu) - \boldsymbol{A}_4^* (\mathbb{E}_\cdot \hat \nu) ), \label{eq:mfg_inf_v} \\ \hat{\nu}_t &= G(t, b^\infty), \label{eq:mfg_inf_nu} 
\end{align}
for all $t \in [0,T]$, where $F$ and $G$ are defined in~\eqref{eq:solmapFG_mfg}.  
\end{theorem}
The proof of Theorem~\ref{thm:meanfieldgame_inf} is deferred to Section~\ref{sec-proof-infinite}.

In light of  Theorems~\ref{thm:meanfieldgame} and~\ref{thm:meanfieldgame_inf}, we readily obtain the following connection between the mean-field game and the infinite-player game:

\begin{remark}
Suppose it holds that
\begin{equation*}
    b^i = \beta^i + \beta^0, \qquad b^\infty =  \mathbb E [\beta] + \beta^0,
\end{equation*}
where $(\beta^i)_{i \in \mathbb{N}}$ are independent copies of the process $\beta$. Then, the generic player's mean-field game equilibrium strategy given by Theorem \ref{thm:meanfieldgame} coincides with the $i$-th player's strategy in the infinite-player mean-field game Nash equilibrium given by Theorem \ref{thm:meanfieldgame_inf}.
\end{remark}

\begin{remark}
The Nash equilibrium of the infinite-player mean-field game in Theorem~\ref{thm:meanfieldgame_inf} is comparable with the Nash equilibrium of the finite-player game presented in Theorem~\ref{thm:main-finite}. The main difference is that the associated Fredholm equations and hence their solutions in~\eqref{eq:solmapFG_mfg},~\eqref{eq:solmapFG_mfg_atilde} and~\eqref{eq:solmapFG_mfg_ahat} become considerably simpler in the infinite-player mean-field limit compared to their counterparts in the finite-player game in~\eqref{eq:opt_ubar},~\eqref{eq:FHubarCoeff},~\eqref{eq:opt_ui} and~\eqref{eq:FHuiCoeff}.  
\end{remark}

\subsection{Convergence and Approximation Results}
\label{subsec:convergence}

The connection between the finite-player Nash equilibrium from Section~\ref{sec:FinitePlayer} and the infinite-player Nash equilibrium from Section~\ref{subsec:infinteplayer} is established in the next convergence theorem.

\begin{theorem} \label{thm:convergence} 
Assume that \eqref{ass:P} as well as Assumptions~\ref{assum-op} and~\ref{ass:inf_Player} are satisfied. For any $N\geq2$, let $(\hat{u}^{i,N})_{i=1,\ldots,N} \subset \mathcal{U}$ denote the Nash equilibrium strategies of the $N$-player game given in~\eqref{eq:opt_ui}. Moreover, let $((\hat{v}^i)_{i \in \mathbb{N}}, \hat{\nu}) \subset \mathcal{U}$ denote the Nash equilibrium of the infinite-player game given in~\eqref{eq:mfg_inf_v} and~\eqref{eq:mfg_inf_nu}. Then, there exists a constant $C>0$ such that
\begin{equation}\label{eq:conv_average}  
\begin{aligned}
 \sup_{0 \leq s \leq T} \mathbb{E} \left[ \left( \hat{\nu}_s - \frac{1}{N} \sum_{i=1}^N \hat{u}^{i,N}_s \right)^2 \right] \leq C  \left(\frac{1}{N^2} \vee h(N) \right), \quad \textrm{for all } N \geq 1, 
\end{aligned}
\end{equation}
as well as 
\begin{equation} \label{eq:conv_policy}  
\begin{aligned}
\sup_{i=1,...,N}\sup_{0 \leq s \leq T} \mathbb{E} \left[ \left( \hat{u}^{i,N}_s - \hat{v}^{i}_s \right)^2 \right] \leq C  \left(\frac{1}{N^2} \vee h(N) \right),\quad \textrm{for all } N \geq 1, 
\end{aligned}
\end{equation}
where $h$ is given in~\eqref{ass:infPlayer_b:eq}.
\end{theorem} 

The proof of Theorem~\ref{thm:convergence} can be found in Section~\ref{sec-proof-infinite}.

In addition, we obtain that the infinite-player Nash equilibrium from Section~\ref{subsec:infinteplayer} provides an approximate Nash equilibrium for the finite-player game from Section~\ref{sec:FinitePlayer} in the sense of Corollary~\ref{thm-eps-nash2} below. To this end, we first recall the definition of an $\eps$-Nash equilibrium.

\begin{definition}
Let $\mathcal U$ denote a class of admissible controls and fix $\eps>0$. A set of controls $\{w^j \in \mathcal U : j =1,\ldots,N\}$ forms an $\eps$-Nash equilibrium with respect to a collection of objective functionals $\{J^j(\cdot,\cdot): j=1,\ldots, N\}$ if it satisfies 
 $$
\sup_{w \in \mathcal U }J^{j}(w; w^{-j}) \leq J^{j}(w^j; w^{-j}) + \eps, \quad \textrm{for all } j=1,\ldots,N.   
$$
\end{definition} 

For any $u\in \mathcal U$ we
introduce the following norm
\be \label{2-t-norm}
\|u\|_{2,T}= \left(\int_0^T \E[u_t^2]dt \right)^{1/2}. 
\ee
In the following we assume that  
\be \label{assum-b-i} 
\sup_{i\in \mathbb{N}} \| b^i \|_{2,T} < \infty. 
\ee

We have following approximation result.

\begin{theorem} \label{thm-eps-nash}
Suppose that \eqref{ass:P}, \eqref{assum-b-i} as well as Assumptions~\ref{assum-op} and~\ref{ass:inf_Player} are satisfied. For any $N \geq 2$, let $J^{i,N}$ be the performance functional of player $i \in \{1,\ldots,N\}$ in the $N$-player game in
\eqref{eq:J^i}. Moreover, let $(\hat v^i)_{i \in \mathbb{N}}$ denote the equilibrium strategies of the mean-field game from Theorem \ref{thm:meanfieldgame_inf} and let $\hat{v}^{-i}=(\hat{v}^{1},\ldots,\hat{v}^{i-1},\hat{v}^{i+1},\ldots,\hat{v}^{N})$ for any $1 \leq i \leq N$. Then, for all $u \in \mathcal U$ there exists a constant $C>0$ independent of $N$ and~$u$ such that for all $i \in \{1, \ldots, N\}$ we have  
\begin{equation*} 
J^{i,N}(u; \hat{v}^{-i}) \leq J^{i,N}(\hat{v}^{i}; \hat{v}^{-i})+ C  (1\vee \|u\|^2_{2,T}) (h(N)^{1/2} \vee N^{-1}). 
\end{equation*}
\end{theorem} 

The proof of Theorem~\ref{thm-eps-nash} is deferred to Section~\ref{sec-proof-infinite}.

From Theorem~\ref{thm-eps-nash} and Theorem \ref{thm:convergence} it follows that we can define a subclass of admissible strategies $\mathcal U_{b} \subset \mathcal U$ that includes $\{ (\hat{u}^{1,N}, \ldots, \hat{u}^{N,N}) : N \in \mathbb{N} \}$ from Theorem~\ref{thm:main-finite} and $(\hat{v}^i)_{i \in \mathbb{N}}$ from Theorem~\ref{thm:meanfieldgame_inf} such that
$$
\sup_{u \in\mathcal U_{b} }  \|u\|^2_{2,T} <\infty. 
$$
Then the following corollary follows immediately from Theorem \ref{thm-eps-nash}. 

\begin{corollary} [$\eps$-Nash equilibrium] \label{thm-eps-nash2} 
Suppose that \eqref{ass:P} as well as Assumptions~\ref{assum-op} and~\ref{ass:inf_Player} are satisfied. For any $N \geq 2$, let $J^{i,N}$ be the performance functional of player $i \in \{1,\ldots,N\}$ in the $N$-player game in
\eqref{eq:J^i}. Moreover, let $(\hat v^i)_{i \in \mathbb{N}}$ denote the equilibrium strategies of the mean-field game from Theorem \ref{thm:meanfieldgame_inf} and let $\hat{v}^{-i}=(\hat{v}^{1},\ldots,\hat{v}^{i-1},\hat{v}^{i+1},\ldots,\hat{v}^{N})$ for any $1 \leq i \leq N$. Then, for all $u \in \mathcal U_b$ there exists a constant $C>0$ independent of $N$ such that 
$$
J^{i,N}(\hat{v}^{i}; \hat{v}^{-i}) \leq \sup_{u \in \mathcal U_{b}} J^{i,N}(u; \hat{v}^{-i}) \leq J^{i,N}(\hat{v}^{i}; \hat{v}^{-i}) + O\left((h(N)^{1/2} \vee N^{-1}) \right). 
$$
\end{corollary}

\section{Stochastic Fredholm equations}\label{sec-fredholm} 

In this section we derive general results on linear stochastic Fredholm equations of the second kind, which appear in the first order condition of the $N$-player game and the mean-field game. Our results on the {semi-explicit} solution and the stability of the solution to this class of equations are central ingredients in the proofs of the main results in Sections \ref{sec:FinitePlayer} and \ref{sec:MeanFieldGame}. All the solutions to the equations considered in this section are referred to as strong solutions. 

Recall that for any $G \in \mathcal G$ we define 
$$
G_t(s,r)=G(s,r)  1_{\{s,r\geq t\}}, 
$$
and we denote by $ {\boldsymbol{G}}_t$ the associated Volterra operator.  

The following proposition derives the solution to stochastic Fredholm equations of the second kind.
 \begin{proposition}\label{L:FredholmConditional}
Let $K,L:[0,T]^2 \to \mathbb R$ be two Volterra kernels in $\mathcal{G}$ such that the operators
\be \label{d-opt} 
 {\boldsymbol{D}}_t = \id  + \boldsymbol{K}_t + \boldsymbol{L^*}_t
\ee
are invertible for all $0\leq t \leq T$ and satisfy
\begin{equation} \label{bounded-dD} 
\sup_{t\leq T} \|\boldsymbol{D}^{-1}_t \|_{\rm{op}}  < \infty. 
\end{equation}
Let $(f_t)_{0 \leq t \leq T}$ be a progressively measurable process such that $\int_0^T \mathbb E[f_t^2] dt<\infty$. Then, the linear stochastic Fredholm equation 
\be \label{v-eq} 
v_t = f_t - \int_0^t K(t,r) v_r dr -\int_t^T L(r,t) \mathbb E_t v_r dr, \quad t\in [0,T],
\ee
admits a unique progressively measurable solution $(v_t)_{0 \leq t \leq T}$ in $\mathcal U$ given by 
\begin{align}
    v_t =   \left((\id - \boldsymbol B)^{-1} a \right) (t), \quad 0\leq t \leq T,  
\end{align}
where
\begin{equation} \label{eq:FHubarCoeff2}
\begin{aligned} 
a_t &=   f_t  -  \langle  1_{\{t< \cdot\}}  L(\cdot,t),  {\boldsymbol D}_t^{-1}1_{\{t< \cdot\}} \E_t f_{\cdot} \rangle_{L^2}\\
B(t,s) &=  1_{\{s<t\}}  \left( \left\langle  1_{\{t< \cdot\}}   L(\cdot,t), {\boldsymbol D}_t^{-1}   1_{\{t\leq \cdot\}}  K(\cdot,s)  \right  \rangle_{L^2}    -   K(t,s)   \right),
\end{aligned}
\end{equation}
and $\boldsymbol B$ is an integral operator induced by the kernel $B$.
\end{proposition}
The proof of Proposition \ref{L:FredholmConditional} is postponed to the end of this section. 
 
We define operator norm as follows, 
\be \label{op-norm} 
 \|\boldsymbol{G}\|_{\rm {op}}= \sup_{f \in L^2([0,T],\R)} \frac{\|\boldsymbol G f\|_{L^2}}{\|f\|_{L^2}},
\ee
for an operator $\boldsymbol{G}$ from $ L^2\left([0,T],{\R}\right)$ to itself.

The following proposition derives a stability result for stochastic Fredholm equations of the second kind. 
 \begin{proposition} \label{L:Fredholm_conv}
Let $K^N,L^N, L, K:[0,T]^2 \to \mathbb R$ be Volterra kernels in $\mathcal{G}$ such that the operators 
\be \label{d-opt-gen} 
 {\boldsymbol{D}}^N_t = \id +   \boldsymbol{K}^N_t + (\boldsymbol{L}_t^N)^* 
\ee
and 
$\boldsymbol{D}_t$ given in \eqref{d-opt}   are invertible for all $0\leq t\leq T$ and satisfy 
\begin{equation} \label{bounded-d} 
\sup_{t\leq T} \|\boldsymbol{D}^{-1}_t \|_{\rm{op}} +\sup_{N\geq 1}  \sup_{t\leq T}  \|(\boldsymbol{D}^{N}_t)^{-1}\|_{\rm{op}}   < \infty. 
\end{equation}
Let $f^N ,f $ be progressively measurable processes on the probability space $(\Omega, \mathcal F,(\mathcal F_t)_{0 \leq t\leq T}, \P)$, with sample paths in  $L^2([0,T],\mathbb R)$, for $N=1,2,\ldots$
Assume further that there exist bounded functions $h_1, h_2: \mathbb{R}_+ \mapsto \mathbb R_+$ such that $\lim_{x \rr \infty} h_i(x) = 0$ and 
 \be \label{ker-con} 
 \begin{aligned} 
\sup_{t\leq T} \int_0^T \big(K^N(t,s) -K(t,s) \big)^2 ds & \leq h_1(N),  \quad \textrm{for all } N \geq 1,\\
\sup_{s\leq T} \int_0^T \big(L^N(t,s) -L(t,s) \big)^2 dt & \leq h_1(N), \quad \textrm{for all } N \geq 1,\\
 \end{aligned} 
 \ee
 and 
 \be \label{f-con} 
\sup_{t\leq T} \E\big[ (f^N_t -f_t)^2\big]  \leq h_2(N). 
\ee
For each $N\geq 1$ let $v^N$ be the   solution to \eqref{v-eq} with $(K^N,L^N,f^N)$. Then there exists a constant $C>0$ such that  
\be 
 \sup_{t\in [0,T]} \E\big[ (v^N_t -v_t)^2\big] \leq C( h_1(N)+ h_2(N) ),  \quad  \textrm{for all } N \geq 1,
\ee
where $v$ is the solution of \eqref{v-eq} with $(K,L,f)$. 
\end{proposition}

\begin{proof} 
Define $d^N_t = v^N_t -v_t$. Then from \eqref{eq:FHubar2} it follows that 
\be \label{diff}
\begin{aligned} 
d^N_t &= a^N_t -a_t + \int_0^t B^N(t,s)v^N_s ds-  \int_0^t B(t,s)v_s ds \\ 
&= a^N_t -a_t + \int_0^t B^N(t,s)d^N_s ds-  \int_0^t (B(t,s) - B^N(t,s))v_s ds, \\ 
\end{aligned}
\ee
where $a$ and $B$ are as in \eqref{eq:FHubarCoeff2} and 

\begin{equation} \label{eq:FHubarCoeff2-N}
\begin{aligned} 
a^N_t &=   f^N_t  -  \langle1_{\{t < \cdot\}}  L^N(\cdot, t),  ({\boldsymbol D}^N_t)^{-1}1_{\{t < \cdot\}} \E_t f^N_{\cdot} \rangle_{L^2}\\
B^N(t,s) &=  1_{\{s<t\}}  \left( \left\langle 1_{\{t \leq \cdot\}} L^N(\cdot, t), ({\boldsymbol D}_t^N)^{-1}   1_{\{t\leq \cdot\}}  K^N(\cdot,s)  \right  \rangle_{L^2}    -   K^N(t,s)   \right).
\end{aligned}
\end{equation}
It follows from \eqref{bounded-d}   and \eqref{ker-con} that 
\be \label{d-opt-con} 
 \sup_{t\leq T} \|{\boldsymbol{D}}^N_t - {\boldsymbol{D}}_t \|_{\textrm{op}}   \leq 2(h_1(N))^{1/2}, \quad \textrm{for all } N\geq 1. 
\ee
We observe that 
\be \label{gam-id} 
 (\boldsymbol {D}_t^N  )^{-1}
- ( \boldsymbol {D}_t  )^{-1}
= (\boldsymbol {D}^N_t  )^{-1}\left( \boldsymbol {D}_t 
-  \boldsymbol {D}^N_t \right)( \boldsymbol {D}_t  )^{-1}.
\ee 
By taking the operator norm on both sides of~\eqref{gam-id} 
and using \eqref{bounded-d}  and~\eqref{d-opt-con} it follows that 
\be \label{D-inv-con} 
\begin{aligned}
   \sup_{t\leq T} \left\|(\boldsymbol {D}_t^N  )^{-1}
- ( \boldsymbol {D}_t^{}  )^{-1}\right\|_{\rm op} 
&\le C \sup_{t\leq T} \| \boldsymbol {D}_t^N -  \boldsymbol {D}_t^{}\|_{\rm op}\\
 & \leq C(h_1(N))^{1/2},\quad \textrm{for all } N\geq 1.
\end{aligned}
\ee
From \eqref{eq:FHubarCoeff2}, \eqref{ker-con}, \eqref{f-con}, \eqref{eq:FHubarCoeff2-N}, \eqref{D-inv-con}, and several applications of Cauchy-Schwartz inequality we get for all $N \geq 1$,
\be \label{f-con2} 
    \sup_{t\leq T} \E\big[ (a^N_t -a_t)^2\big]  \leq C( h_1(N)+ h_2(N)), 
\ee
and 
\be \label{B-con2} 
  \sup_{t\leq T} \int_0^T \big(B^N(t,s) -B(t,s) \big)^2 ds \leq C h_1(N).
\ee
Recall the notation 
$$
 \|v\|_{\mathcal H^2}  = \E \left[ \int_0^T v_s^2 ds \right]. 
$$
By plugging \eqref{f-con2} and \eqref{B-con2} into \eqref{diff} and using Cauchy-Schwarz inequality, we get for all $N$ sufficiently large, 
\be \label{diff2}
\begin{aligned} 
\E[(d^N_t)^2]  
& \leq C \Bigg( \E[(a^N_t -a_t)^2] + \E\left[ \left( \int_0^t B^N(t,s)d^N_s ds \right)^2 \right] \\
& \qquad + \E \left[ \left(   \int_0^t (B^N(t,s) - B(t,s))v_s ds  \right)^2 \right] \Bigg)\\ 
& \leq C \Bigg(  h_1(N)+ h_2(N) + \E \left[ \int_0^t(d^N_s)^2 ds  \right]  \int_0^t(B^N(t,s))^2 ds\\
& \qquad +  \|v_s\|_{\mathcal{H}^2}  \int_0^t (B^N(t,s) - B(t,s))^2 ds \Bigg)\\ 
& \leq C \left( h_1(N)+ h_2(N)  +h_1(N) \E \left[ \int_0^t(d^N_s)^2 ds  \right] \right). 
 \end{aligned}
\ee
Since $C$ above is not depending on $N$ and $t$ we get using Fubini's theorem, 
\be \label{diff3}
 \sup_{s \leq t} \E[(d^N_s)^2]   \leq C \left( h_1(N)+ h_2(N)  +h_1(N)  \int_0^t \sup_{r \leq  s} \E[(d^N_r)^2] ds \right), \quad \textrm{for all } 0\leq t \leq T. 
 \ee
Since $h_i$ are bounded functions, from Gronwall lemma it follows that 
\be \label{diff4}
\sup_{s \in [0, T]} \E[(v^N_s-v_s)^2]  =    \sup_{s \in [0, T]} \E[(d^N_s)^2]   \leq  C(T)  ( h_1(N)+ h_2(N) ), \quad \textrm{for all } N\geq 1.
 \ee
 and this completes the proof. 
\end{proof}

\begin{proof} [Proof of Proposition \ref{L:FredholmConditional}] 
We derive the solution to \eqref{v-eq}. For fixed $t \in [0,T]$, we define the process 
\begin{align*}
  m_t(s) := 1_{\{t\leq s\}} \E_t[v_s], \qquad s \in [0,T].
\end{align*}
Taking conditional expectation $\mathbb{E}_t[\cdot]$ in the linear Volterra equation in~\eqref{v-eq}, multiplying by $1_{\{t \leq s\}}$ and using the tower property of conditional expectation gives us 
\begin{align}\label{eq:mbar2}
  m_t(s) = & \, 1_{\{t\leq s\}} \E_t f_s - 1_{\{t\leq s\}}\int_0^t K(s,r) v_r dr  - 1_{\{t\leq s\}} \int_t^s  K(s,r) \mathbb{E}_t v_r dr \nonumber \\ 
    & \, - 1_{\{t\leq s\}} \int_s^T  L(r,s) \E_t v_r dr \nonumber \\
    = & \, f^{v}_t(s) \,- 1_{\{t\leq s\}} \int_t^s  K(s,r) \mathbb{E}_t v_r dr  - 1_{\{t\leq s\}} \int_s^T  L(r,s) \E_t v_r dr \nonumber \\
     = & \, f^{v}_t(s) \,-  \int_t^s  K_t(s,r)  m_t(r) dr  -  \int_s^T  L^*_t(s,r)  m_t(r) dr \nonumber \\
    = & \, f^{v}_t(s) - \left( \boldsymbol{K}_t + \boldsymbol{L^*}_t   \right)(  m_t)(s), \qquad s \in [0,T], 
\end{align}
where
\begin{equation} \label{f-eq2} 
    f^{v}_t(s) := 1_{\{t\leq s\}} \E_t f_s - 1_{\{t\leq s\}}\int_0^t K(s,r)v_r dr,
\end{equation}
and where we used the definition $K_t(s,r)=K(s,r)\mathds 1_{r\geq t}$. 

Recall that $ {\boldsymbol{D}}_t$ was defined in \eqref{d-opt}.  Using the invertibility assumption on $ {\boldsymbol{D}}_t$ and \eqref{bounded-dD}, we can invert the equation \eqref{eq:mbar2} to get 
\begin{equation} \label{bar-m-sol2}
    m_t(s) = ( {\boldsymbol{D}}_t^{-1} f^{v}_t)(s), \quad s \in [0,T].
\end{equation}
Plugging \eqref{bar-m-sol2} into~\eqref{v-eq} we obtain
\begin{equation} \label{eq:ubar2} 
\begin{aligned} 
  v_t
     =& \,  f_t - \int_0^t  K(t,r)v_r dr   - \int_t^T     L(r,t)   ( {\boldsymbol{D}}_t^{-1} f^{v}_t)(r)dr. 
\end{aligned}
\end{equation}
Next, we focus on the third term on the right-hand side of \eqref{eq:ubar2} and re-express it as a linear functional in $v$. Using \eqref{f-eq2}  we get 
\be \label{dec-g-d} 
\begin{aligned}
   &  \int_t^T L(r,t) ( {\boldsymbol D}_t^{-1}f^{v}_t )(r) dr \\
   &= \langle  1_{\{t \leq \cdot\}} L(\cdot,t), {\boldsymbol D}_t^{-1}f^{ v}_t \rangle_{L^2} \\ 
     &= - \left\langle 1_{\{t \leq \cdot\}} L(\cdot,t), {\boldsymbol D}_t^{-1}   1_{\{t\leq \cdot\}}\int_0^t K(\cdot,r)    v_r dr \right \rangle_{L^2}  +  \langle 1_{\{t \leq \cdot\}}L(\cdot,t),  {\boldsymbol D}_t^{-1}1_{\{t \leq \cdot\}} \E_t f_{\cdot} \rangle_{L^2} \\
 &=- \int_0^t \left\langle  1_{\{t \leq \cdot\}}  L(\cdot,t), {\boldsymbol D}_t^{-1}   1_{\{t\leq \cdot\}}  K(\cdot,r)  \right  \rangle_{L^2}v_r dr  + \langle 1_{\{t \leq \cdot\}} L(\cdot,t),  {\boldsymbol D}_t^{-1}1_{\{t \leq \cdot\}} \E_t f_{\cdot} \rangle_{L^2} . 
\end{aligned}
\ee
Inserting the {obtained} expression back into \eqref{eq:ubar2} yields the following Fredholm equation for $v$:
\begin{align} \label{eq:FHubar2} 
v_t = a_t + \int_0^t B(t,s)  v_s ds, 
\end{align}
with $a$ and $B$ as in \eqref{eq:FHubarCoeff2}. 

It is left to argue that \eqref{eq:FHubar2} has a solution.  Note that \eqref{eq:FHubar2} is a linear Volterra equation which admits a solution for any fixed $\omega \in \Omega$, whenever $a(\omega)\in L^2([0,T],\mathbb R)$ and $B$ satisfies $$\sup_{t\leq T} \int_0^T B(t,s)^2 ds < \infty. $$  Indeed, the solution is given in terms of the resolvent  $R^B$ of $B$,  which exists by virtue of Corollary 9.3.16 in \cite{gripenberg1990volterra} and satisfies
\begin{align*}
\int_0^T \int_0^T |R^B(t,s)| dt ds < \infty.
\end{align*}
In this case, the solution $v$ to~\eqref{eq:FHubar2} is given by
 \begin{align*}
    v_t = a_t + \int_0^t R^B(t,s) a_s ds. 
 \end{align*}
Note that $R^B$ is the kernel induced by the operator given by 
 $$
\boldsymbol R^B =  (\id - \boldsymbol B)^{-1} - {\id}. 
$$
One would still need to check that $u^* \in \mathcal U$ defined as in \eqref{def:admissset}. This follows from the following lemma. 

\begin{lemma} \label{bnd-coef} 
Let $v$ be as in \eqref{v-eq} and $a, B$ as in \eqref{eq:FHubarCoeff2}. Then the following hold: 
\begin{itemize} 
\item[\bf{(i)}]  $\mathbb E\left[ \int_0^Ta_s^2 ds \right] <\infty$, 
\item[\bf{(ii)}] $\sup_{t\leq T} \int_0^T B(t,s)^2 ds < \infty, $ 
\item[\bf{(iii)}] {$\mathbb E\left[\int_0^T (v_s)^2 ds\right] < \infty$.}  
\end{itemize} 
\end{lemma} 
The proof of Lemma \ref{bnd-coef} is similar to the proof of Lemma 7.1 in \cite{abi2022optimal}; hence, it is omitted. 

It follows that the equation \eqref{eq:FHubar2} admits a unique solution $v \in \mathcal{U}$ which is given by 
$$
    v_t = \left((\id - \boldsymbol B)^{-1} a \right) (t), \quad 0\leq t \leq T,  
$$
with $a$ given in \eqref{eq:FHubarCoeff2} and $\boldsymbol B$ is the integral operator induced by the kernel $B$ in \eqref{eq:FHubarCoeff2}.
\end{proof}

 \section{Proofs of Theorems~\ref{thm:opt_ubar} and~\ref{thm:main-finite}} \label{sec-proof-finite}
 
This section is dedicated to the proofs of Theorems~\ref{thm:opt_ubar} and~\ref{thm:main-finite}. 
In order to ease the notation, we omit the upper case $N$ from $u^N, \bar u^N, u^{i,N}, J^{i,N}$, etc., throughout this section. 

We first establish the strict concavity property of $u^{i} \mapsto J^{i}(u^{i};u^{-i})$ which is crucial for the derivation of the Nash equilibrium.

\begin{lemma} \label{lem:concave_finite} 
Let $i \in \{1,\ldots, N\}$. Then, under Assumption~\ref{assum-op}, for any $u^{-i} \in \mathcal{U}^{N-1}$ fixed, the functional $u^{i} \mapsto J^{i}(u^{i};u^{-i})$ in~\eqref{eq:J^i} is strictly concave in $u^{i} \in \mathcal{U}$.
\end{lemma}

\begin{proof}
Fix $i \in \{1,\ldots, N\}$, let $u^i \in \mathcal{U}$ and $u^{-i} \in \mathcal{U}^{N-1}$. First, using the notation from~\eqref{def:uNotation}, we  have
 \begin{align*}
\langle \bar u,\boldsymbol{A}_1 \bar u \rangle_{L^2} = & \langle \frac{u^i}{N} + \bar u^{-i}, \boldsymbol{A}_1 (\frac{u^i}{N} + \bar u^{-i}) \rangle_{L^2} \\
= & \langle \frac{u^i}{N}, \boldsymbol{A}_1 \frac{u^i}{N} \rangle_{L^2} + \langle \frac{u^i}{N}, \boldsymbol{A}_1 \bar u^{-i} \rangle_{L^2} + \langle \bar u^{-i}, \boldsymbol{A}_1 \frac{u^i}{N} \rangle_{L^2}+ \langle \bar u^{-i}, \boldsymbol{A}_1 \bar u^{-i} \rangle_{L^2} \\
 = & \langle \frac{u^i}{N}, \boldsymbol{A}_1 \frac{u^i}{N} \rangle_{L^2} + \frac{1}{N} \langle  u^i , (\boldsymbol{A}_1 +\boldsymbol{A}_1^*)\bar u^{-i} \rangle_{L^2} + \langle \bar u^{-i}, \boldsymbol{A}_1 \bar u^{-i} \rangle_{L^2}. 
\end{align*}
Hence, we can rewrite the objective functional in~\eqref{eq:J^i} as follows 
\begin{equation}\label{eq:Jlemma}
\begin{aligned} 
   & J^{i}(u^i; u^{-i}) \\
    & = \E\Big[ - \langle \bar u,\boldsymbol{A}_1 \bar u \rangle_{L^2} - \langle u^i,\boldsymbol{A}_2 u^i \rangle_{L^2} - \langle u^i,(\boldsymbol{A}_3 + \boldsymbol{A}_4^* ) \bar u \rangle_{L^2} \Big. \\
    & \Big. \hspace{30pt} + \langle b^i, u^i \rangle_{L^2} + \langle b^0, \bar u\rangle_{L^2} + c^i \Big] \\
    & = \E\Big[-\langle u^i, \big(\frac{\boldsymbol{A}_1 }{N^2} + \frac{\boldsymbol{A}_3 +  \boldsymbol{A}_4^*}{N}+\boldsymbol A_2 \big) u^i \rangle_{L^2} - \langle u^i, ( \frac{\boldsymbol{A}_1+\boldsymbol{A}_1^*}{N} + \boldsymbol{A}_3 + \boldsymbol{A}_4^*)\bar u^{-i} \rangle_{L^2} \\
    & \hspace{30pt} - \langle \bar u^{-i},\boldsymbol A_1 \bar u^{-i} \rangle_{L^2} +
   \langle b^i +  \frac{1}{N} b^0, u^i \rangle_{L^2} + \langle b^0, \bar u^{-i} \rangle_{L^2} + c^i \Big]. 
\end{aligned}
\end{equation}  
Then, for any $w^i \in \mathcal{U}$ such that $w^i \neq u^i$, $d\P \otimes dt$-a.e.~on $\Omega \times [0,T]$, and for all $\varepsilon \in (0,1)$, a direct computation yields 
\begin{equation}
    \begin{aligned}
        & J^{i}(\varepsilon u^{i} + (1-\varepsilon) w^{i} ; u^{-i}) - \varepsilon J^{i}(u^{i};  u^{-i}) - (1-\varepsilon) J^{i}(w^{i};  u^{-i}) \\
        & = \varepsilon (1-\varepsilon) \E\left[ \langle u^i - w^i , \big(\frac{\boldsymbol{A}_1 }{N^2} +\frac{ \boldsymbol{A}_3 + \boldsymbol{A}_4^*}{N}+\boldsymbol{\hat{A}}_2  + \lambda \id  \big) (u^i - w^i) \rangle_{L^2} \right] \\        
        & = \varepsilon (1-\varepsilon) \left( \E\left[ \langle u^i - w^i ,\big(\frac{\boldsymbol{A}_1 }{N^2} +\frac{ \boldsymbol{A}_3 + \boldsymbol{A}_4^*}{N}+\boldsymbol{\hat{A}}_2 \big) (u^i - w^i) \rangle_{L^2} \right] + \lambda \E[ \| u^i- w^i\|_{L^2}^2 ] \right) \\ 
        & = \varepsilon (1-\varepsilon) \left( \E\left[ \langle u^i - w^i ,\big(\frac{\boldsymbol{A}_1 }{N^2} + \frac{\boldsymbol{A}_3 + \boldsymbol{A}_4}{N}+\boldsymbol{\hat{A}}_2 \big) (u^i - w^i) \rangle_{L^2} \right] + \lambda \E[ \| u^i- w^i\|_{L^2}^2 ] \right) \\
        & = \varepsilon (1-\varepsilon) \left( \E\left[ \langle u^i - w^i , \boldsymbol{G} (u^i - w^i) \rangle_{L^2} \right] + \lambda \E[ \| u^i- w^i\|_{L^2}^2 ] \right) \\
        & > 0,
    \end{aligned}
\end{equation}
where we used  the fact that $\| u^i- w^i\|_{L^2}^2 >0$, $\P$-a.s., Remark \ref{rem:opNonNeg} and \eqref{b-bar}, and the definition of the adjoint $\langle u, \boldsymbol{A}_4^* u \rangle_{L^2} = \langle u, \boldsymbol{A}_4 u \rangle_{L^2}$.
\end{proof}

Let $i \in \{1,...,N\}$ and $u^{-i} \in \mathcal U^{N-1}$. Lemma \ref{lem:concave_finite} {states} that the map $J^i(\cdot; u^{-i}): u^i \mapsto J^i(u^i; u^{-i})$ of player $i$'s best response to all other players' fixed strategies $u^{-i}$ is strictly concave. Hence, it admits a unique maximiser characterised by the critical point at which the G\^ateaux derivative
\begin{align*}
    \langle  \nabla  J^i(u^i; u^{-i}) , h  \rangle = \lim_{\epsilon \to 0} \frac{ J^i(u^i + \epsilon h; u^{-i}) -  J^i(u^i; u^{-i})}{\epsilon},
\end{align*}
vanishes for all $h\in \mathcal U$. Therefore, a control $u^i \in \mathcal U$ {maximises} \eqref{eq:J^i} if and only if $u^i$ satisfies the first order condition 
\begin{align} \label{foc} 
 \langle \nabla J^i(u^i; u^{-i}) , h  \rangle  = 0, \quad h \in \mathcal U,
\end{align}
(see Proposition 2.1 of \cite[Chapter II]{EkelTem:99}).

Using \eqref{eq:Jlemma} and \eqref{def:operator} we obtain by an application of Fubini's theorem, that 
\begin{align}\label{eq:FOCC}
&\langle \nabla  J^i(u^i; u^{-i}) , h  \rangle   \\
&=\mathbb E\left[ \langle h,  \,-2\lam u^i  -  (\boldsymbol{G} + \boldsymbol{G}^*)(u^i) -\boldsymbol{H} (\bar{u}^{-i})+ b^i +  \frac{1}{N} b^0 \rangle_{L^2}    \right]
\\
&= \int_0^T \mathbb E\left[ h_t \left(   \,-2\lam u^i_t  -  (\boldsymbol{G} + \boldsymbol{G}^*)(u^i)(t) - \boldsymbol{H}( \bar{u}^{-i})(t) + b^i_t + \frac{1}{N} b^0_t \right) \right] dt,
\end{align} 
for any $h \in \mathcal U$.
By conditioning on $\mathcal F_t$ and using the  tower property we get from \eqref{def:uNotation}, \eqref{foc} and \eqref{eq:FOCC} the following first order condition, 
\begin{equation} \label{eq:FOCN1}
\begin{aligned} 
&-2\lam u^i_t  -  (\boldsymbol{G} + \boldsymbol{G}^*)(\E_t u^i)(t) - \boldsymbol{H}(\E_t \bar{u}^{-i})(t)   +   b^i_t + \frac{1}{N} b^0_t \\ 
&\quad = \, -2\lam u^i_t-  \left(\boldsymbol{G} + \boldsymbol{G}^* -\frac{1}{N} \boldsymbol{H} \right)(\E_t u^i)(t) - \boldsymbol{H}(\E_t \bar{u})(t) + b^i_t + \frac{1}{N} b^0_t 
\\
&\quad =0,
\quad d\P \otimes dt \textrm{-a.e.~on } \Omega \times [0,T].
\end{aligned}
\end{equation}
It follows that for the Nash equilibrium the following must hold for all $i \in \{1,\ldots,N\}$, 
\begin{equation} \label{eq:FOCN2}
\begin{aligned}
    2 \lambda u_t^i = & \,  b^i_t + \frac{1}{N} b^0_t - \int_0^t \left(\frac{A_1}{N} + A_3\right)(t,r) \bar u_r dr - \int_t^T \left( \frac{A_1 }{N} + A_4 \right)(r,t)  \E_t \bar  u_r dr  \\&- \int_0^t \left(   G(t,r)-\frac{1}{N}  \left(\frac{A_1}{N}+A_3 \right)(t,r)\right) u^i_r  dr  \\ & - \int_t^T     \left(   G(r,t)-\frac{1}{N}  \left( \frac{A_1}{N} + A_4 \right)(r,t)\right)  \E_t u^i_r dr ,
\end{aligned}
\end{equation}
$d\P \otimes dt \textrm{-a.e.~on } \Omega \times [0,T]$.
By summing over $i$ in~\eqref{eq:FOCN2}, using \eqref{b-bar-def} and scaling by $1/N$ we get the following equation for $\bar u$: 
\begin{align}
    2 \lambda \bar u_t = & \,  \bar b_t + \frac{1}{N} b^0_t  - \int_0^t \left( \left(\frac{A_1}{N}+A_3\right)(t,r) + G(t,r) - \frac{1}{N} \left(\frac{A_1}{N}+A_3\right)(t,r) \right) \bar u_r dr \nonumber \\ 
    & \, - \int_t^T   \left(  G(r,t)  +  \left(\frac{A_1}{N} + A_4 \right) (r,t)-\frac{1}{N} \left( \frac{A_1}{N} + A_4 \right) \right)  \E_t \bar u_r dr \nonumber \\
     =& \,  \bar b_t + \frac{1}{N} b^0_t - \int_0^t \left( \frac{N-1}{N} \left(\frac{A_1}{N}+A_3\right)(t,r) + G(t,r)   \right) \bar u_r dr \nonumber \\
     & \, - \int_t^T    \left( \frac{N-1}{N} \left(\frac{A_1}{N} + A_4 \right) (r,t) + G(r,t)   \right)  \E_t \bar u_r dr,   \label{eq:FOCaggregate} 
\end{align}
$d\P \otimes dt \textrm{-a.e.~on } \Omega \times [0,T]$. After establishing \eqref{eq:FOCaggregate} we are now ready to prove Theorem~\ref{thm:opt_ubar}.

\begin{lemma} \label{new-lem}
Let $\boldsymbol{K}$ be an admissible Volterra operator as in Definition \ref{def-admis-op}. Then for any $\eps>0$, $\eps\id+\boldsymbol{K}_t$ is invertible and  
$\|(\eps\id+\boldsymbol  K_t)^{-1}\|_{op} \leq 1/\eps$.
\end{lemma} 
\begin{proof}
Property~\eqref{pos-def} implies for any $f\in L^2\left([0,T],\mathbb R\right)$, 
\begin{equation} \label{gt1}
     \langle f, \boldsymbol K f\rangle_{L^2} = \frac{1}{2}\langle f, (\boldsymbol K + \boldsymbol K^*) f \rangle_{L^2}  \geq 0.
\end{equation}
Let $(\lam,f)$ be an eigen-pair  of $\eps\id +\boldsymbol K_t$. Then $\lam \|f\|_{L^2}^2 =  \langle f, (\eps\id +\boldsymbol K_t) f\rangle_{L^2} > 0$ by \eqref{gt1} and the fact that $
\boldsymbol K + \boldsymbol K^*$ is a   
 symmetric non-negative definite operator.  The proof of the bound $\|(\eps\id+\boldsymbol  K_t)^{-1}\|_{op} \leq 1/\eps$ is similar to the proof of Lemma 7.2 in \cite{abi2022optimal} since 
$$
\langle f, (\eps\id +\boldsymbol K_t)  f \rangle \geq \langle f,\eps \id f \rangle  \geq \eps\|f\|, \quad \textrm{for all } f \in L^2([0,T],\mathbb R). 
$$
\end{proof}

\begin{proof}[Proof of Theorem \ref{thm:opt_ubar}]
We now turn into the derivation of \eqref{eq:opt_ubar}. We apply Proposition \ref{L:FredholmConditional} to~\eqref{eq:FOCaggregate} for $\bar u$ instead of $v$. Specifically we have 
\begin{equation} \label{f-k-l-1} 
\begin{aligned}
f_t = & \, \frac{1}{2\lambda} \left(\bar b_t +\frac{1}{N} b^0_t\right) \qquad\qquad\qquad\qquad\qquad \\
2 \lambda K = & \, \frac{N-1}{N} \left(\frac{A_1}{N}+A_3\right) + G , \\ 
2 \lambda L = & \, \frac{N-1}{N} \left( \frac{A_1}{N} + A_4 \right) + G.
\end{aligned}
\end{equation}
In this case we have that ${\boldsymbol{D}}_t$ in \eqref{d-opt} is of the form $\frac{1}{2\lambda} \ol {\boldsymbol{D}}_t$ with  
\be
 \ol {\boldsymbol{D}}_t = 2\lam(\id +   \boldsymbol{K}_t + \boldsymbol{L^*}_t) = 2\lam\id + \frac{N-1}{N} \boldsymbol{H}_t + \boldsymbol{G}_t+ \boldsymbol{G}^*_t. 
\ee
Note that $ \ol  {\boldsymbol{D}}_t$ is invertible and satisfies \eqref{bounded-dD}  since $\ol {\boldsymbol{D}}_t + \ol {\boldsymbol{D}}^*_t$ is  positive semi-definite (see Lemma \ref{new-lem}). From Assumption \ref{assum-op} and \eqref{def:operator} it follows that $G \in \mathcal G$, therefore ${K}, {L} \in \mathcal G$. Hence, together with \eqref{ass:P} it follows that the assumptions of Proposition \ref{L:FredholmConditional} are satisfied and we have 

\be 
\begin{aligned}
\bar {u}_{t}  = \left((\id -  \ol {\boldsymbol B})^{-1} \ol {a} \right) (t), \quad 0\leq t \leq T.
\end{aligned}
\ee
where 
\begin{equation} 
\begin{aligned} 
\ol a_t & := \frac 1{2\lambda} \left(\bar b_{t} +\frac{1}{N} b^0_t  -  \left\langle  1_{\{t\leq \cdot\}} \ol L(\cdot,t), \ol{\boldsymbol D}_t^{-1} 1_{\{t\leq \cdot\}} \E_t \left[\bar{b}_{\cdot} + \frac{1}{N} b^0_{\cdot}\right] \right\rangle_{L^2} \right), \\
\ol B(t,s) & :=  1_{\{s<t\}}\frac 1{2\lambda} \bigg( \left\langle   1_{\{t\leq \cdot\}} \ol L(\cdot,t),\ol{\boldsymbol D}_t^{-1}   1_{\{t\leq \cdot\}}  \ol K(\cdot,s)    \right  \rangle_{L^2}   -  \ol K(t,s)    \bigg),
\end{aligned}
\end{equation}
with $\ol K:= 2\lambda K$ and $\ol L:= 2\lambda L$ for $K, L$ as in \eqref{f-k-l-1}. This agrees with \eqref{eq:opt_ubar} and completes the proof. 
\end{proof} 

\begin{proof}[Proof of Theorem \ref{thm:main-finite}]
Given the solution $\bar u$ from Theorem \ref{thm:opt_ubar}, we can now continue with solving the first order condition~\eqref{eq:FOCN2} separately for each $u^i$, $i \in \{1,\ldots,N\}$ and hence derive a Nash equilibrium to the game \eqref{eq:J^i}. 

We apply Proposition \ref{L:FredholmConditional} to~\eqref{eq:FOCN2} for $u^i$ instead of $v$. Specifically we have
\be \label{f-k-l-22} 
\begin{aligned}          
f_t & = \frac{1}{2\lambda} \Bigg( b^i_t + \frac{1}{N} b^0_t - \int_0^t \left(\frac{A_1(t,r)}{N} + A_3(t,r) \right) \bar u_r dr \Bigg. \\
& \hspace{38pt} \Bigg. - \int_t^T \left( \frac{A_1(r,t)}{N} + A_4(r,t) \right) \E_t \bar  u_r dr  \Bigg), 
\end{aligned}
\ee
From \eqref{f-k-l-22}, \eqref{def:operator} and \eqref{b-bar} {we can apply Proposition \ref{L:FredholmConditional} to }
\be \label{l-opt22} 
\begin{aligned} 
 2\lambda K = G - \frac{1}{N} \left(\frac{A_1}{N}+A_3\right), \quad  2\lambda  L = G -   \frac{1}{N} \left( \frac{A_1}{N} +A_4 \right).
 \end{aligned} 
\ee
By Assumption \ref{assum-op} it  follows that $A_1, A_3, A_4, G \in \mathcal G$, therefore $K, L \in \mathcal G$. Hence the assumptions of Proposition \ref{L:FredholmConditional} are satisfied. Note that ${\boldsymbol{D}}_t$ in Proposition \ref{L:FredholmConditional} is of the form $\frac{1}{2\lambda} \wt{\boldsymbol{D}}_t$ with 
\be \label{d-opt2} 
\begin{aligned} 
 \wt{\boldsymbol{D}}_t &:=   2\lambda\left( \id+   \boldsymbol{K}_t + \boldsymbol{L}^*_t \right) = 2\lambda \id - \frac{1}{N}\boldsymbol{H}_t +\boldsymbol{G}_t+ \boldsymbol{G}^*_t.
 \end{aligned} 
\ee
Note that $ \wt  {\boldsymbol{D}}_t$ is invertible and satisfies \eqref{bounded-dD}  since $ \wt {\boldsymbol{D}}_t +  \wt {\boldsymbol{D}}^*_t$ is  positive semi-definite (see Lemma \ref{new-lem}). Using Proposition \ref{L:FredholmConditional} we obtain for all $i \in \{ 1,\ldots,N\}$ the representation 
$$
\begin{aligned}
 \hat{u}^{i}_{t} =\left((\id -  \boldsymbol{B})^{-1} a^i \right) (t), \quad 0\leq t \leq T, 
\end{aligned}
$$
where
 \begin{equation} 
\begin{aligned}
a^i_t & := \frac 1{2\lambda} \bigg(b^i_{t} +\frac{1}{N} b^0_t - \E_t\left[(\boldsymbol{H}\bar u)(t)\right]  \\
&\qquad \qquad \qquad -  \left\langle{   1_{\{t\leq \cdot\}} \hat L(\cdot,t)},  \wt{\boldsymbol D}_t^{-1} 1_{\{t\leq \cdot\}} \E_t \left[b^i_{\cdot} +\frac{1}{N} b^0_{\cdot} - ( \boldsymbol{H} \bar u)(\cdot) \right] \right\rangle_{L^2} \bigg), \\
B(t,s) & :=  1_{\{s<t\}}\frac 1{2\lambda} \bigg(  \left\langle   1_{\{t\leq \cdot\}} \hat L(\cdot,t),  \wt {\boldsymbol D}_t^{-1}   1_{\{t\leq \cdot\}}     \hat K(\cdot, s) \right  \rangle_{L^2} -\hat K (t, s)    \bigg), 
\end{aligned}
\end{equation}
where $\hat K = {2\lambda}K$ and $\hat L = {2\lambda}L$. 
This agrees with \eqref{eq:opt_ui} and \eqref{eq:FHuiCoeff}.  

Next, we argue for uniqueness of the Nash equilibrium. Assume that $(w^1,\ldots, w^N)$ in $\mathcal U^N$ is another Nash equilibrium. Then, by Definition~\ref{def:Nash} it  holds that
\begin{equation} \label{eq:twoNEs2}
   \quad J^{i}(w^{i};w^{-i}) \geq J^{i}(v;w^{-i}), \quad \textrm{for all } v\in \mathcal{U}, \, i \in \{1,\ldots,N\}.
\end{equation}
In particular, for a fixed $i \in \{1,\ldots, N\} $, $w^i$ maximizes the strictly concave functional $w\to J^{i}(w; w^{-i})$ (recall Lemma \ref{lem:concave_finite}). This means that \eqref{foc} is satisfied with $w$ in place of $u$.  It follows that $\bar w$ satisfies the equation \eqref{eq:FOCaggregate} with $\bar u$ replaced by $\bar w$. By uniqueness of the solution to the Fredholm equation  \eqref{eq:FOCaggregate}, see Proposition~\ref{L:FredholmConditional}, we get $\bar w = \bar u$, $d\mathbb P \otimes dt-a.e.$ on $\Omega\times [0,T]$. Similarly, we obtain  $w^i=v^i$ from the uniqueness of solutions to the Fredholm equation \eqref{eq:FOCN2} where  we replaced $u$ by $w$.  This concludes the proof of uniqueness. 
\end{proof} 

\section{Proofs of the results from Section~\ref{sec:MeanFieldGame}}
\label{sec-proof-infinite}

Similar to the finite-player game in Section~\ref{sec:FinitePlayer}, we have the following strict concavity result.

\begin{lemma} \label{lem:concave_mfg}
Under Assumption~\ref{assum-op}, for any $\mu \in \mathcal{U}^{0}$, the functional $v \mapsto J(v;\mu)$ in~\eqref{eq:J_generic} is strictly concave in $v \in \mathcal{U}$.
\end{lemma}

\begin{proof}
This follows similar as in the proof of Lemma~\ref{lem:concave_finite} above. Let $v \in \mathcal{U}$ and $\mu \in \mathcal{U}^{0}$. Then, for any $w \in \mathcal{U}$ such that $w \neq v$ $d\P \otimes dt$-a.e.~on $\Omega \times [0,T]$ and for all $\varepsilon \in (0,1)$ a direct computation, combined with the definition of $\boldsymbol{A}_2$ in \eqref{b-bar},  shows that 
\begin{equation}
\begin{aligned}
& J(\varepsilon v + (1-\varepsilon) w ; \mu) - \varepsilon J(v; \mu) - (1-\varepsilon) J(w; \mu) \\
& = \varepsilon (1-\varepsilon) \E\left[ \langle v - w , \boldsymbol{A}_2 (v - w) \rangle_{L^2} \right] \\
& = \varepsilon (1-\varepsilon) \Big( \E\left[ \langle v - w , \boldsymbol{\hat{A}}_2 (v - w) \rangle_{L^2} \right] + \lambda \E[ \| v- w\|_{L^2}^2 ] \Big) \\
& > 0, 
\end{aligned}
\end{equation}
since $\E[ \| v- w \|_{L^2}^2] >0$ and the operator $\boldsymbol{\hat{A}}_2$ is nonnegative definite; recall Remark~\ref{rem:opNonNeg}.
\end{proof}

We are now ready to prove Theorem~\ref{thm:meanfieldgame}.

\begin{proof}[Proof of Theorem \ref{thm:meanfieldgame}]
We start with rewriting the cost functional $J$ in~\eqref{eq:J_generic} as
\begin{equation*}
\begin{aligned}
    J(v; \mu) = & \, \E\left[ -\langle \mu,\boldsymbol A_1  \mu \rangle_{L^2} - \langle v,\boldsymbol A_2 v \rangle_{L^2} - \langle v,(\boldsymbol{A}_3 + \boldsymbol{A}^*_4) \mu \rangle_{L^2} \right. \\
    & \left. \quad\, + \langle b, v \rangle_{L^2} + \langle b^0, \mu \rangle_{L^2} + c \right] \\
    = & \, \E\left[ -\langle v, \boldsymbol A_2 v \rangle_{L^2} + \langle \tilde{b}, v \rangle_{L^2} + \tilde{c} \right], 
\end{aligned}
\end{equation*}
where
\begin{equation}
\begin{aligned} \label{eq:bitilde}
\tilde{b} & = b - (\boldsymbol{A}_3 + \boldsymbol{A}^*_4) \mu, \\
\tilde{c} & =  \langle b^0, \mu \rangle_{L^2} + c -\langle \mu,\boldsymbol A_1  \mu \rangle_{L^2}. 
\end{aligned} 
\end{equation}
By the strict concavity obtained in Lemma \ref{lem:concave_mfg} and similarly to the proof of Theorem~\ref{thm:opt_ubar} in the finite-player game, observe that for a fixed process $\mu \in \mathcal{U}^0$ the first order condition for the generic player’s best response is equivalent to
\begin{align}
    0 &= -(\boldsymbol A_2 + \boldsymbol A_2^*)(\E_t v)(t) + \E_t \tilde{b}_t \nonumber  \\
    &= - 2\lambda v_t -  (\boldsymbol{\hat{A}}_2 )(v)(t) - (\boldsymbol{\hat{A}}^*_2)(\E_t v)(t) +  \E_t \tilde{b}_t \label{eq:FOCmfgEq1}\\
    &= - 2 \lambda v_t - \int_0^t  \hat{A}_2(t,s) v_s ds - \int_t^T \hat{A}_2(s,t) \E_t v_s ds + \E_t \tilde{b}_t, \nonumber
    \quad \textrm{for all } t\in [0,T],
\end{align}
where we used the fact that $\hat{A}_2$ is a Volterra kernel. In other words,  $\hat{v}$ is optimal in~\eqref{eq:meanfield_opt} (with $\mu \in \mathcal{U}^0$ fixed) if and only if $\hat{v}$  satisfies
\begin{equation}
\begin{aligned} \label{eq:proofMFG1}
2\lambda \hat{v}_s = & \, {b}_s  - \int_0^s A_3(s,r)  \mu_r dr - \int_s^T A_4(r,s)  \mathbb{E}_s \mu_r dr \\ & \, - \int_0^s  \hat{A}_2(s,r) \hat{v}_r dr - \int_s^T \hat{A}_2(r,s) \E_s \hat{v}_r dr.
\end{aligned}
\end{equation}
Taking conditional expectation with respect to ${\mathcal F^{0}_T}$, using the independence between $\beta$ and $\beta^0$, the fact that $\mu$ is $\mathbb F^0$-progressively measurable and \eqref{def:b_generic}, we get 
\begin{align}
2\lambda \mathbb E[{\hat v}_s | \mathcal F^{0}_{T}] = & \, \beta^0_s + \mathbb E[\beta_s]  - \int_0^s A_3(s,r)  \mu_r dr - \int_s^T A_4(r,s)  \mathbb{E}_s [\mu_r] dr  \\ 
&  - \int_0^s  \hat{A}_2(s,r)  \mathbb E[{\hat v}_r | \mathcal F^{0}_{T}] dr 
- \int_s^T \hat{A}_2(r,s) \mathbb E_s[\mathbb E[{\hat v}_r| \mathcal F^{0}_{T}]] dr. \label{eq:proofMFG2}
\end{align}
Together with the consistency condition in~\eqref{eq:meanfield_cons}, this suggests a candidate equation for $\hat \mu$ given by 
\begin{align} \label{eq:proofMFG3}
2\lambda \hat \mu_s = & \,  \beta^0_s {+ \mathbb E[\beta_s] }  - \int_0^s \left(A_3(s,r)  + \hat A_2(s,r) \right) \hat\mu_r dr \\
& - \int_s^T  \left( \hat{A}_2(r,s) + A_4(r,s) \right) \mathbb E_s[{\hat \mu}_r]  dr.
\end{align}
An application of Proposition~\ref{L:FredholmConditional} yields the existence of a solution $\hat \mu$ to \eqref{eq:proofMFG3} given by \eqref{eq:mfgnu}. In particular, $\hat \mu \in \mathcal U^{0}$ as desired. Moreover, inserting $\hat \mu$ into equation \eqref{eq:proofMFG1} and applying once more Proposition~\ref{L:FredholmConditional} yields that the corresponding $\hat v$ is of the form \eqref{eq:mfgv} and belongs to $\mathcal U$. This proves that $(\hat v, \hat \mu)$ satisfy the requirement \eqref{eq:meanfield_opt}. To justify that the consistency condition~\eqref{eq:meanfield_cons} is satisfied, simply  observe that equation \eqref{eq:proofMFG2} with  $\hat \mu$ in place of $\mu$ shows that $(\mathbb E[\hat v_s | \mathcal F^{0}_T])_{s\leq T}$ solves the same equation as $\hat \mu$ in \eqref{eq:proofMFG3}. By uniqueness of the solution to the Fredholm equation, we deduce the consistency condition \eqref{eq:meanfield_cons} is indeed satisfied. Uniqueness of the  mean-field game equilibrium follows from the uniqueness of the corresponding stochastic Fredholm equations that characterize the optimum, along the lines of the uniqueness proof of Theorem \ref{thm:main-finite}.
\end{proof}

Next, we address the proof of Theorem~\ref{thm:meanfieldgame_inf} in the infinite-player game formulation in Section~\ref{subsec:infinteplayer}. First, analogous to Lemma~\ref{lem:concave_mfg}, we have the following:

\begin{lemma} \label{lem:concave_inf}
Let $i \in \mathbb{N}$. Under Assumption~\ref{assum-op}, for any $\nu \in \mathcal{U}$, the functional $v^i \mapsto J^{i,\infty}(v^i;\nu)$ in~\eqref{eq:J_inf} is strictly concave in $v^i \in \mathcal{U}$.
\end{lemma}

\begin{proof}
This follows as in the proof of Lemma~\ref{lem:concave_mfg} above.
\end{proof}

We are now ready to prove Theorem~\ref{thm:meanfieldgame_inf}

\begin{proof}[Proof of Theorem~\ref{thm:meanfieldgame_inf}]
Similar to equation~\eqref{eq:proofMFG1} in the proof of Theorem~\ref{thm:meanfieldgame}, we obtain that for fixed $\nu \in \mathcal{U}$ the first order condition for player $i$'s best response is given by
\begin{align}\label{eq:vi_FOC} 
2\lambda v^i_s = & \, b^i_s - \int_0^s A_3(s,r)  \nu_r dr - \int_s^T A_4(r,s) \mathbb{E}_s[ \nu_r] dr  \\
& - \int_0^s  \hat{A}_2(s,r) v^i_r dr - \int_s^T \hat{A}_2(r,s) \E_s v^i_r dr.
\end{align}
Solving~\eqref{eq:vi_FOC} for $v^i$ (with $\nu \in \mathcal{U}$ fixed) is an application of Proposition~\ref{L:FredholmConditional} and we get for all $i \in \mathbb{N}$ the representation
\begin{equation} \label{vi_FOC2}
    v^i_t = F(t,  b^i - \boldsymbol{A}_3 (\nu) - \boldsymbol{A}_4^* (\mathbb{E}_\cdot \nu) )
\end{equation}
Next, we claim that the infinite-player Nash equilibrium's mean-field strategy $\hat{\nu}$ must satisfy the linear Volterra equation 
\begin{equation}\label{eq:FOC_nu}
\begin{aligned}
    2\lambda \hat{\nu}_s = & \,  b^\infty_s  - \int_0^s  \left( \hat A_2(s,r)  + {A}_3(s,r) \right) \hat{\nu}_r dr \\
    & - \int_s^T \left( \hat{A}_2(r,s) + A_4(r,s) \right) \mathbb{E}_s[\hat\nu_r]  dr,
\end{aligned}
\end{equation}
where $b^\infty$ denotes the limit from Assumption~\ref{ass:inf_Player}. By virtue of Proposition~\ref{L:FredholmConditional}, the solution to~\eqref{eq:FOC_nu} is given by
\begin{equation} \label{eq:FOC_nu2}
    \hat{\nu}_t = G(t, b^{\infty}).
\end{equation}
Then, plugging~$\hat{\nu}$ in~\eqref{eq:FOC_nu2} back into~\eqref{vi_FOC2} above will give us the desired mean-field game equilibrium strategies in~\eqref{eq:mfg_inf_v} and~\eqref{eq:mfg_inf_nu} in the sense of Definition~\ref{def:Nash_Inf}, which we denote with $( (\hat{v}^i)_{i \in \mathbb{N}},\hat{\nu}) \subset \mathcal{U}$.

In order to justify our claim, it remains to show that the consistency condition in~\eqref{def:consistency_MFG} is satisfied by $\hat{\nu}$ characterized by~\eqref{eq:FOC_nu}. To this end, we define for all $N \in \mathbb{N}$ the averages 
\begin{equation}
\hat \nu^N := \frac 1 N \sum_{i=1}^N \hat v^i.  \end{equation}
Note that taking the average over $i$ in \eqref{eq:vi_FOC} (with $\hat{v}^i$ and $\hat{\nu}$) yields that $\hat \nu^N$ solves the Volterra equation
\begin{equation} \label{eq:vi_averages}
\begin{aligned}
    2\lambda {\hat \nu^N}_s = & \, \frac{1}{N} \sum_{i=1}^N b^i_s - \int_0^s A_3(s,r)\hat{\nu}_r dr - \int_s^T A_4(r,s) \mathbb{E}_s[ \hat{\nu}_r] dr \\   
    & \, - \int_0^s  \hat{A}_2(s,r) {\hat \nu^N}_r dr - \int_s^T \hat{A}_2(r,s) \E_s {\hat \nu^N}_r dr.
\end{aligned}
\end{equation}
Hence, together with Assumption~\ref{ass:inf_Player}, the convergence in the consistency condition in~\eqref{def:consistency_MFG} follows from an application of the stability property of Fredholm equations derived in Proposition~\ref{L:Fredholm_conv} to equations \eqref{eq:vi_averages} and \eqref{eq:FOC_nu}. Uniqueness of the  infinite-player game equilibrium follows from the uniqueness of the corresponding stochastic Fredholm equations that characterize the optimum, along the lines of the uniqueness proof of Theorem \ref{thm:main-finite}.
\end{proof}

We finish with providing the proofs of Section~\ref{subsec:convergence} and start with the convergence result in Theorem~\ref{thm:convergence}.

\begin{proof}[Proof of Theorem \ref{thm:convergence}] 
We only prove~\eqref{eq:conv_average} as the proof of the convergence in~\eqref{eq:conv_policy} follows the same lines. From~\eqref{eq:FOCaggregate} above we get that 
\begin{equation*}
\bar u^N_t :=\frac{1}{N} \sum_{i=1}^N \hat{u}^{i,N}_t    
\end{equation*}
satisfies the Volterra equation in~\eqref{v-eq} with 
\begin{equation} \label{eq:proof_conv1}
\begin{aligned}
f^N_t = & \, \frac{1}{2\lam}  \left( \bar b_t + \frac{1}{N} b^0_t \right), \\ 
K^N = & \, \frac{1}{2\lam} \left( \frac{N-1}{N} \left(\frac{A_1}{N}+A_3\right) + G \right) , \\  L^N = & \,  \frac{1}{2\lam} \left( \frac{N-1}{N} \left(\frac{A_1}{N} + A_4 \right) + G \right).
\end{aligned}
\end{equation} 
Recall from~\eqref{b-bar-def} that $\bar b_t = \frac{1}{N} \sum_{i=1}^N b^i_t$. From \eqref{eq:FOC_nu} it follows that $\hat \nu$ satisfies \eqref{v-eq} with 
\begin{equation} \label{eq:proof_conv2} 
f_t = \frac{1}{2\lam}  b^{\infty}_t, \quad K =   \frac{1}{2\lam} \left(  \hat{A}_2 + A_3\right), \quad L =\frac{1}{2\lam} \left( \hat{A}_2 + A_4 \right). 
\end{equation}
Emphasizing the dependence of $G$ defined in~\eqref{def:operator} on $N$ we write $G^N$. From \eqref{def:operator} and \eqref{eq:assumtionG} it follows that there exists $C>0$ such that 
\begin{equation} \label{eq:proof_conv3} 
 \sup_{0\leq t\leq T} \int_0^T\big(G^N(t,s) - \hat{A}_2 (t,s)\big)^2ds  \leq C(N^{-2}),  \quad \textrm{for all } N \geq 1.
\end{equation} 
Moreover, from~\eqref{eq:assumtionG}, \eqref{eq:proof_conv1},~\eqref{eq:proof_conv2} and~\eqref{eq:proof_conv3} we get that
\begin{equation}  
\begin{aligned} 
 \sup_{0\leq t\leq T} \int_0^T \big(K^N(t,s) -K(t,s) \big)^2 ds  & \leq C(N^{-2}), \quad \textrm{for all } N \geq 1, \\
 \sup_{0\leq s\leq T} \int_0^T \big(L^N(t,s) -L(t,s) \big)^2 dt  & \leq C(N^{-2}), \quad \textrm{for all } N \geq 1,
\end{aligned} 
\end{equation}
and Assumption~\ref{ass:inf_Player} implies
\begin{equation}
\begin{aligned}
\sup_{t\leq T} \E\big[ (f^N_t -f_t)^2\big] \leq h(N), \quad \textrm{for all } N \geq 1. 
\end{aligned}
\end{equation}
Therefore, the convergence rate in~\eqref{eq:conv_average} follows by an application of Proposition~\ref{L:Fredholm_conv}. 
\end{proof} 

Next, in order to prove Theorem~\ref{thm-eps-nash} the first ingredient is following auxiliary result. 

\begin{lemma}  \label{lemma-con-mf} 
Assume that \eqref{ass:P} as well as Assumptions~\ref{assum-op} and~\ref{ass:inf_Player} are satisfied. Let $((\hat{v}^i)_{i\in \mathbb{N}}, \hat{\nu})$ be the equilibrium strategies of the mean-field game in the sense of Definition~\ref{def:Nash_Inf}. Then we have     
\begin{equation} \label{lemma-con-mf_eq} 
\sup_{0 \leq t \leq T} \E \left[\left( \frac{1}{N}\sum_{i=1}^N \hat v_t^i - \hat \nu_{t} \right)^{2}\right]= O \left(h(N)\right),
\end{equation}
where $h$ is given in~\eqref{ass:infPlayer_b:eq}.
\end{lemma} 

\begin{proof}
From~\eqref{eq:FOC_nu} it follows that $\hat \nu$ satisfies \eqref{v-eq} with 
\begin{equation} \label{eq:proof_conv22} 
f_t = \frac{1}{2\lam} \left( b^{\infty}_t - \boldsymbol{A}_3(\hat \nu)(t) - \boldsymbol{A}_4^* (\mathbb{E}_t \hat \nu)(t) \right), \quad K = L = \frac{1}{2\lam} \hat{A}_2,
\end{equation}
and from~\eqref{eq:vi_FOC} it follows that $\frac{1}{N}\sum_{i=1}^N \hat v_t^i$ satisfies \eqref{v-eq} with 
\begin{equation} \label{eq:proof_conv23} 
 f^N_t = \frac{1}{2\lam} \left( \frac{1}{N} \sum_{i=1}^N b^{i}_t - \boldsymbol{A}_3(\hat \nu)(t) - \boldsymbol{A}_4^* (\mathbb{E}_t \hat \nu)(t) \right), \quad K = L = \frac{1}{2\lam} \hat{A}_2. 
\end{equation}
By applying Proposition~\ref{L:Fredholm_conv} using \eqref{ass:infPlayer_b:eq} we get \eqref{lemma-con-mf_eq}. 
\end{proof}

The second ingredient is Lemma~\ref{lem-j-dif} below which provides a bound on the difference between the performance functional $J^{i,\infty}$ of the mean-field game in~\eqref{eq:J_inf} and the $N$-player game's performance functional $J^{i,N}$ in \eqref{eq:J^i}.

\begin{lemma}\label{lem-j-dif} 
Suppose that \eqref{ass:P}, \eqref{assum-b-i} as well as Assumptions~\ref{assum-op} and~\ref{ass:inf_Player} are satisfied. Let $((\hat v^i)_{i\in \mathbb{N}}, \hat \nu)$ be the equilibrium strategies of the mean-field game in the sense of Definition~\ref{def:Nash_Inf}. For any $N \geq 2$ and $1 \leq i \leq N$, define  $\hat{v}^{-i}=(\hat{v}^{1},\ldots,\hat{v}^{i-1},\hat{v}^{i+1},\ldots,\hat{v}^{N})$. Then, there exists a constant $C>0$ independent of $N$ such that for all $u \in \mathcal U$ and $i \in \mathbb{N}$ we have
$$
\left| J^{i,N}(u,\hat{v }^{-i}) -  {J}^{i,\infty}(u,\hat \nu)\right| \leq  C  (1+\|u\|^2_{2,T}) \left(h(N)^{1/2}\vee N^{-1} \right). 
$$
\end{lemma} 

\begin{proof}
From \eqref{assum-b-i} and Assumption \ref{ass:inf_Player} and by repeating the same steps as in Proposition \ref{L:Fredholm_conv} we get
\be \label{unif-bnd-v-i}
\sup_{i\in \mathbb{N}}\|\hat{v}^{i}\|_{2,T} < \infty. 
\ee

Let $u\in \mathcal U$. From \eqref{eq:J^i} and \eqref{eq:J_inf} we get 
\be \label{j-sub}
\begin{aligned} 
&\left| J^{i,N}(u; \hat{v}^{-i}) -  J^{i,\infty}(u ;\hat \nu)\right| \\
&\leq \left| \E  \left [\left\langle \frac{1}{N}\left(\sum_{j\not = i} \hat{v}^{j}+u \right),\boldsymbol A_1 \left(\frac{1}{N}\big(\sum_{j\not = i} \hat{v}^{j}+u \big)\right)\right \rangle_{L^2} - \langle\hat \nu,\boldsymbol A_1  \hat \nu \rangle_{L^2}  \right ] \right|  \\
& \quad +\left| \E \left  [ \left \langle   u, (\boldsymbol A_3 + \boldsymbol A_4^*) \left(\frac{1}{N} \left(\sum_{j\not = i} \hat{v}^{j}+u \right) - \hat \nu \right)  \right\rangle_{L^2} \right ] \right|  \\ 
& \quad + \left| \E \left [\left \langle b^0, \frac{1}{N} \left(\sum_{j\not = i} \hat{v}^{j}+u \right)- \hat \nu \right \rangle_{L^2} \right ] \right|  \\
 &=: I_1 + I_2 + I_3.
\end{aligned} 
\ee
Using \eqref{2-t-norm}, H\"older inequality and the standard operator norm bound, we get for the second term~$I_2$ in~\eqref{j-sub}:
 \be \label{i-2}
\begin{aligned} 
I_2 & = \left| \E \left  [ \left \langle   u, (\boldsymbol A_3 + \boldsymbol A_4^*)   \left(\frac{1}{N} \left(\sum_{j\not = i} \hat{v}^{j}+u \right) - \hat \nu \right)  \right\rangle_{L^2} \right ] \right| \\
 &\leq C \|u\|_{2,T} \| \boldsymbol A_3 + \boldsymbol A_4^* \|_{op} \left \| \frac{1}{N} \left(\sum_{j\not = i} \hat{v}^{j}+u \right) - \hat \nu \right\|_{2,T} \\ 
  &\leq C  \|u\|_{2,T}   \left(   \left\| \frac{1}{N}\sum_{j =1}^N \hat{v}^{j}    - \hat \nu  \right \|_{2,T}  +\frac{1}{N}  \left\|  \hat{v}^{i} - u \right \|_{2,T}  \right) \\ 
   &\leq C\|u\|_{2,T}(h(N)^{1/2} \vee N^{-1})       \left(1 +\|\hat{v}^{i}\|_{2,T} + \|u\|_{2,T}   \right) \\
   & \leq C(h(N)^{1/2} \vee N^{-1})(1+\|u\|^2_{2,T})  , 
\end{aligned} 
\ee
where we used Lemma \ref{lemma-con-mf} and \eqref{unif-bnd-v-i} in the last two inequalities. Similarly, using once more Fubini's Theorem, Jensen and H\"older inequalities we get for the third $I_3$ in~\eqref{j-sub}:
\be\label{i-3}
\begin{aligned} I_3 & = \left| \E \left [\left \langle b^0, \frac{1}{N} \left(\sum_{j\not = i} \hat{v}^{j}+u \right)- \hat \nu \right \rangle_{L^2} \right ] \right|        \\
&\leq \| b^0 \|_{2,T}   \left\| \frac{1}{N} \left(\sum_{j\not = i} \hat{v}^{j}+u \right)   - \hat \nu  \right \|_{2,T}     \\
&\leq \| b^0 \|_{2,T} \left(   \left\| \frac{1}{N}\sum_{j =1}^N \hat{v}^{j}    - \hat \nu  \right \|_{2,T}  +\frac{1}{N}  \left\|  \hat{v}^{i} - u  \right \|_{2,T}  \right)   \\
&\leq  C(h(N)^{1/2} \vee N^{-1}) \| b^0 \|_{2,T}   \left(1 +\|\hat{v}^{i}\|_{2,T} + \|u\|_{2,T}   \right)  \\ 
&\leq C(h(N)^{1/2} \vee N^{-1}) \| b^0 \|_{2,T}   \left(1 + \|u\|_{2,T}   \right)  ,
\end{aligned} 
\ee
where we used again Lemma \ref{lemma-con-mf} and \eqref{unif-bnd-v-i} in the last two inequalities.
 
Lastly, for the first term $I_1$ in~\eqref{j-sub} we obtain 
 \be \label{i-1}
\begin{aligned} 
I_1 = & \, \left| \E  \left [\left\langle \frac{1}{N}\left(\sum_{j\not = i} \hat{v}^{j}+u \right),\boldsymbol A_1 \left(\frac{1}{N}\big(\sum_{j\not = i} \hat{v}^{j}+u \big)\right)\right \rangle_{L^2} - \langle\hat \nu,\boldsymbol A_1  \hat \nu \rangle_{L^2}  \right ] \right|  \\
\leq & \, \left| \E  \left [\left\langle \frac{1}{N}\left(\sum_{j\not = i} \hat{v}^{j}+u \right),\boldsymbol A_1 \left(\frac{1}{N}\big(\sum_{j\not = i} \hat{v}^{j}+u \big)- \hat \nu\right)\right \rangle_{L^2} \right ] \right|  \\
& + \left| \E  \left [\left\langle \frac{1}{N}\left(\sum_{j\not = i} \hat{v}^{j}+u \right)- \hat \nu ,\boldsymbol A_1  \hat \nu \right \rangle_{L^2} \right ] \right|.
\end{aligned} 
\ee
One can repeat similar steps as in \eqref{i-2} and \eqref{i-3} by using Lemma \ref{lemma-con-mf} and \eqref{unif-bnd-v-i} in order to get 
\be\label{i-1.2}
\begin{aligned} 
I_1 &\leq   C(h(N)^{1/2} \vee N^{-1}) \left(1 + \|u\|^2_{2,T} \right). 
\end{aligned}
\ee
Then, from \eqref{i-2}, \eqref{i-3} and  \eqref{i-1.2} we get the desired upper bound in~\eqref{j-sub}. 
\end{proof}

We are now ready to prove Theorem~\ref{thm-eps-nash}.

\begin{proof} [Proof of Theorem \ref{thm-eps-nash}]
First, note that the inequality 
$$
J^{i,N}(\hat{v}^{i}; \hat{v}^{-i}) \leq \sup_{u \in \mathcal U} J^{i,N}(u; \hat{v}^{-i}) 
$$
holds trivially by the definition of the supremum. Next, by virtue of Lemma~\ref{lem-j-dif} we get for any $u\in \mathcal U$ the upper bound
\be \label{rt1}
\begin{aligned}
J^{i,N}(u; \hat v^{-i})& \leq  J^{i,\infty}(u;\hat \nu) + C(h(N)^{1/2} \vee N^{-1}) (1+\|u\|^2_{2,T}) \\ 
&\leq   J^{i,\infty}(\hat v^i; \hat\nu) + C(h(N)^{1/2} \vee N^{-1}) (1+\|u\|^2_{2,T})  , 
\end{aligned}
\ee
where $C>0$ is a constant not depending on $u$ or $N$ and we used $J^{i,\infty}(\hat v^i;\hat \nu)= \sup_{u\in \mathcal U} J^{i,\infty}(u;\hat \nu)$ in the second inequality. Moreover, using Lemma \ref{lem-j-dif} and \eqref{unif-bnd-v-i} again we get for some constant $\tilde C>0$ the upper bound 
\be \label{rt2} 
\begin{aligned}
  J^{i,\infty}(\hat v^i; \hat \nu )  &\leq   J^{i,N}(\hat v^i; \hat v^{-i}) + C(h(N)^{1/2} \vee N^{-1})  (1+\|\hat v^i\|^2_{2,T})  \\
  &\leq J^{i,N}(\hat v^i; \hat v^{-i}) +\tilde C(h(N)^{1/2} \vee N^{-1}) ,
\end{aligned}
\ee
where we used Lemma \ref{lemma-con-mf} and the fact that $\hat v^i \in \mathcal U$ in the second inequality. Finally, using \eqref{rt2} to bound the right hand side of \eqref{rt1}, we get 
$$
 J^{i,N}(u; \hat v^{-i})  \leq  J^{i,N}(\hat v^i; \hat v^{-i}) + C(h(N)^{1/2} \vee N^{-1})  (1\vee  \|u\|^2_{2,T})
$$
for all $u \in \mathcal U$ which completes the proof. 
\end{proof}

\appendix

\section*{Acknowledgments}
We would like to thank Fabrice Djete and to the anonymous referees for a number of useful comments and suggestions that significantly improved this paper.

\end{document}